\documentclass{article}

\usepackage{microtype} 
\usepackage{graphicx}
\usepackage{subcaption}
\usepackage{booktabs} 

\usepackage{hyperref}

\usepackage[preprint]{icml2024}

\usepackage{amsmath}
\usepackage{amssymb}
\usepackage{mathtools}
\usepackage{amsthm}
\usepackage{dsfont}

\usepackage{enumerate}

\usepackage{tablefootnote}

\usepackage[capitalize,noabbrev]{cleveref}
\usepackage{ulem}
\theoremstyle{plain}
\newtheorem{theorem}{Theorem}[section]
\newtheorem{proposition}[theorem]{Proposition}
\newtheorem{lemma}[theorem]{Lemma}

\newtheorem{fact}[theorem]{Fact}
\theoremstyle{definition}
\newtheorem{definition}[theorem]{Definition}

\theoremstyle{remark}

\usepackage[textsize=tiny]{todonotes}

\newcommand{\R}[0]{\mathbb{R}}
\newcommand{\C}[0]{\mathbb{C}}
\newcommand{\N}[0]{\mathbb{N}}
\newcommand{\Z}[0]{\mathbb{Z}}

\newcommand{\Prob}[0]{\mathbb{P}}

\newcommand{\E}[0]{\mathbb{E}}
\newcommand{\V}[0]{\mathbb{V}}
\newcommand{\Ind}[0]{\mathds{1}}
\newcommand{\argmin}[0]{\operatorname{argmin}}

\newcommand{\eqdef}{\vcentcolon=}
\newcommand{\defeq}{=\vcentcolon}

\newcommand\mech[1]{{#1}}
\newcommand\vect[1]{\mathbf{#1}}

\newcommand\tv[2]{\mathrm{TV}\left( {#1}, {#2} \right)}
\newcommand\kl[2]{\mathrm{KL}\left( \left. {#1}\right\Vert {#2} \right)}
\newcommand\renyi[3]{\text{D}_{#1}\left( \left. {#2}\right\Vert {#3} \right)}
\newcommand\ham[2]{d_\mathrm{ham}\left( {#1}, {#2} \right)}

\newcommand\p[1]{\left( {#1}\right)}

\newcommand\set[1]{\mathcal{#1}}
\newcommand\proj[0]{\operatorname{Proj}}

\definecolor{ao(english)}{rgb}{0.0, 0.5, 0.0}
\newcommand\green[1]{\textcolor{ao(english)}{#1}}
\newcommand\red[1]{\textcolor{red}{#1}}
\newcommand\orange[1]{\textcolor{orange}{#1}}
\newcommand\blue[1]{\textcolor{blue}{#1}}

\newcommand\vspan[0]{\text{Span}}

\icmltitlerunning{Privately Learning Smooth Distributions on the Hypercube by Projections}

\begin{document}
\def\UrlBreaks{\do\/\do-}

\twocolumn[
\icmltitle{Privately Learning Smooth Distributions on the Hypercube \\ by Projections}
\icmlsetsymbol{equal}{*}

\begin{icmlauthorlist}
\icmlauthor{Clément Lalanne}{1}
\icmlauthor{Sébastien Gadat}{1}
\end{icmlauthorlist}

\icmlaffiliation{1}{Toulouse School of Ecolomics, Université Toulouse 1 Capitole, Toulouse, France}

\icmlcorrespondingauthor{Clément Lalanne}{clement.lalanne@tse-fr.eu}

\icmlkeywords{Privacy, Estimation, Quantiles}

\vskip 0.3in
]
\printAffiliationsAndNotice{}
\begin{abstract}
Fueled by the ever-increasing need for statistics that guarantee the privacy of their training sets, this article studies the centrally-private estimation of Sobolev-smooth densities of probability over the hypercube in dimension $d$. The contributions of this article are two-fold : Firstly, it generalizes the one-dimensional results of \cite{lalanne2023about} to non-integer levels of smoothness and to a high-dimensional setting, which is important for two reasons : it is more suited for modern learning tasks, and it allows understanding the relations between privacy, dimensionality and smoothness, which is a central question with differential privacy. 
Secondly, this article presents a private strategy of estimation that is \emph{data-driven} (usually referred to as \emph{adaptive} in Statistics) in order to privately choose an estimator that achieves a good bias-variance trade-off among a finite family of private projection estimators \emph{without prior knowledge of the ground-truth smoothness $\beta$}. This is achieved by adapting the Lepskii method for private selection, by adding a new penalization term that makes the estimation privacy-aware.
\end{abstract}

\section{Introduction}

Multiple experimental pieces of work have demonstrated that the unrestricted use of data for various learning tasks may cause privacy concerns \citep{narayanan2006break,backstrom2007wherefore,fredrikson2015model,dinur2003revealing,homer2008resolving,loukides2010disclosure,narayanan2008robust,sweeney2000simple,gonon2023sparsity,wagner2018technical,sweeney2002k,carlini2022miafromfirstprinciples}. As a result, formal guarantees have been developed through \emph{differential privacy} \citep{dwork2006calibrating} in order to guarantee that a quantity built on users' data does not leak more information than a given threshold. It is now considered as the gold standard in terms of privacy protection, and it is notably used by Apple~\citep{thakurta2017learning}, Google~\citep{erlingsson2014rappor,46411}, Microsoft~\citep{ding2017collecting} and the US Census Bureau~\citep{DBLP:conf/icde/MachanavajjhalaKAGV08,DBLP:conf/sigmod/HaneyMAGKV17,abowd2018us} among many others. 

Let $f$ be a density of probability on $[0, 1 ]^d$ w.r.t. Lebesgue's measure, and let $X_1, \dots, X_n$ be $n$ i.i.d. random variables with a distribution of probability that admits $f$ as density on $[0, 1 ]^d$. In this article, we will study the estimation of $f$ with a quantity $\hat{f}$ that privately builds on $X_1, \dots, X_n$. The notion of privacy that is adopted in this article is the notion of \emph{central zero-concentrated differential privacy} \cite{dwork2016concentrated,bun2016concentrated} (see \Cref{sectionDifferentialPrivacy}). 

This problem is statistically difficult (in the sense that it requires a lot of data) and suffers from the curse of dimensionality, which means that even without privacy considerations, one must expect an exponential number (in the dimensionality) of data points in order to solve it. Yet, its interest lies in its generality, and in its expressivity. Exploring the effects of privacy on this statistical problem is interesting on a theoretical standpoint, in order to better understand differential privacy, and for the practitioner in order to better decide between this general approach and a different one that incorporates more prior information about the distribution to estimate.

The motivations for this problem are multiple. For instance, learning a density allows learning distributions that are very general, and distributions for which we do not have simple parametric representations. On top of that, learning densities allows learning in a tractable way multimodal distributions and mixture distributions whereas the tractability of alternative methods (e.g. EM) is not always obtainable, even without considering privacy. With a private density estimate, a data analyst may estimate various other interesting statistics without having to see the data again (and hence without having to spend more privacy budget) such as the mean, the median, the different modes, \dots A final application that we can mention is private data generation : If one has access to a private estimate of the density, then one may sample new data by rejection sampling.

The privacy constraint naturally has a cost on the utility of estimators for this task, as with other forms of communication constraints \citep{BarnesHO19,BarnesHO20,AcharyaCST21a,AcharyaCMT21,AcharyaCST21b,AcharyaCFST21}. An important question with differential privacy is to precisely characterize this cost, and to compare it to the incompressible error due to the estimation from samples. In this article, we quantify this trade-off when the density $f$ has a certain level of smoothness $\beta$. Furthermore, we also explain how to privately estimate $f$ when this smoothness level is not accessible to the practitioner, a property of the estimator referred-to as \emph{adaptivity}.

\subsection{Related work}

\begin{table*}[h!]
\caption{Comparison with concurrent work.}
\label{tableComparisonResults}
\vskip 0.15in
\begin{center}
\begin{small}
\begin{sc}
\begin{tabular}{lccccr}
\toprule
Work & Privacy & Dimensionality & Smoothness & Adaptivity & Estimation Rate \\
\midrule
\tiny{\cite{wasserman2010statistical}}   & \red{Fixed}& $ \red{d=1}$ & \green{$\beta \in \p{\frac{1}{2}, +\infty}$} & \red{$\times$} & $\Theta\p{n^{-\frac{2 \beta}{2 \beta + 1}}}$\\
\tiny{\cite{barber2014privacy}}   & \green{Variable} & \green{$d \in \N \setminus \{ 0\}$} & \red{$\beta = 1$} & \red{$\times$} & $\Theta\p{n^{-\frac{2}{2  + d}} + (n \sqrt{\rho})^{- \frac{2}{1+d}})}$\\
\tiny{\cite{lalanne2023about}}   & \green{Variable} & \red{$d = 1$} & \orange{$\beta \in \N \setminus \{ 0\}$} & \red{$\times$} & $\Theta\p{n^{-\frac{2 \beta}{2 \beta  + 1}} + (n \sqrt{\rho})^{- \frac{2 \beta}{\beta+1}})}$\\
\textbf{This Work}   & \green{Variable} & \green{$d \in \N \setminus \{ 0\}$} & \green{$\beta \in \p{0, +\infty}$} & \green{$\surd$} & $\Theta\p{n^{-\frac{2 \beta}{2 \beta  + d}} + (n \sqrt{\rho})^{- \frac{2 \beta}{\beta+d}})}$\\
\bottomrule
\end{tabular}
\end{sc}
\end{small}
\end{center}
In \cite{wasserman2010statistical}, the smoothness is defined in terms of Sobolev ellipsoids. The results are presented under pure differential privacy, which implies concentrated differential privacy. In \cite{barber2014privacy}, the smoothness is expressed in therms of Lipschitz continuity, which is usually assimilated heuristically to $\beta = 1$ in terms of Sobolev spaces. Again, the authors worked under $\epsilon$ pure differential privacy, but we took the liberty to express the results with $\rho = \epsilon^2$ in order to simplify comparisons.
\vskip -0.2in
\end{table*}

\paragraph{Statistics and differential privacy.}
Estimating various quantities under differential privacy has received an increasing amount of attention during the last decade. A non-exhaustive list of references include
\cite{wasserman2010statistical,barber2014privacy,diakonikolas2015differentially,karwa2017finite,bun2019privatehypothesis,bun2021privatehypothesis,kamath2019highdimensional,biswas2020coinpress,kamath2020heavytailed,acharya2021differentially,lalanne:thesis,adenali2021unbounded,cai2021cost,brown2021covariance,cai2021cost,kamath2022improved,lalanne2023private,lalanne2022private,singhal2023polynomial,kamath2023biasvarianceprivacy,kamath2023new}. Most of those references study parametric estimation problems (i.e. estimating a quantity living in a finite-dimensional space), and observe (at a meta level) that the error of estimation can usually be expressed as a function of the sample size ($n$), the dimensionality ($d$), the level of privacy ($\rho$), and various quantities that characterize the regularity of the distribution class (sub-Gaussian, moments, smoothness, ...). Besides, the interesting effects of the privacy can be observed when the level of privacy ($\rho$) is considered as a free variable of the problem. Conversely, fixing the level of privacy usually results in rates of estimation that are the same as in the non-private case. In this article, we will consider the privacy budget as free, thus allowing to investigate some interesting trade-offs between the sample size $n$ and the level of privacy $\rho$.

\paragraph{Unconstrained density estimation.}
The problem of estimating the density $f$ is known as a \emph{nonparametric statistical problem}. It differs from some more usual problems in the sense that the quantity to estimate ($f$) lives in an infinite-dimensional vector space. Specific techniques thus have to be used to estimate it. One of those techniques consists in approximating $f$ by learning its projections on subspaces of growing dimension, and it is being used in this article.
Without privacy concerns, this problem has been extensively studied for multiple decades. Without trying to be exhaustive, some important monographs include \cite{conover1999practicalnonparametric,gyorfi2002distribution,tsybakov2003introduction,wasserman2006allnonparametric}.

\paragraph{Density estimation with differential privacy.}
With differential privacy, the problem of nonparametric density estimation has been studied in a few articles. Before continuing, it is important to note that there are two main privacy attack models in the literature (depending on whether an aggregator can be trusted or not), leading to two distinct definitions of privacy : \emph{central} differential privacy or \emph{local} differential privacy \cite{10.1145/773153.773174,4690986}. This article studies the \emph{central} model, and \emph{local} differential privacy is outside its scope. This paragraph only covers the literature in the central model.
An important early piece of work \cite{wasserman2010statistical} has paved the way for private non-parametric density estimation, presenting general private projection and histogram estimators. However, it only studied the case where the level of privacy $\rho$ is kept constant, leading to the rather anticlimactic conclusion that privacy had no effect on the optimal rate of estimation for the problem at hand. In \cite{barber2014privacy}, the authors were the first to consider $\rho$ as a variable, and to study rates of convergence that are not privacy-agnostic. A shortcoming of their study is that they only study the estimation of Lipschitz-continuous densities, which imposes a fixed level of smoothness. More recently, \cite{lalanne2023about} studied the estimation of one-dimensional densities of general integer-valued Sobolev-smoothness $\beta$ in a non privacy-agnostic way. This is the piece of work that is the closest to our article. However, three problems are that the authors only tackle the case of one-dimensional data, that the smoothness parameter only takes discrete values, and that their optimal estimation procedure needs to know the ground-truth smoothness $\beta$ beforehand. This article solves all of these issues. A comparison between our article and this body of literature is summarized in \Cref{tableComparisonResults}.

\paragraph{Adaptive estimation.} 

Classical frameworks for adaptive estimators build estimators of the bias of each model and select the model with the lowest estimated squared bias penalized by the variance \cite{Akaike1998,mallows1973,Birg1993RatesOC,Barron1999RiskBF,laurent2000adaptive,massart2007}.
The  Lepskii method
\cite{lepskii1991adaptivegaussiannoise,lepskii1992adaptive1,lepskii1992adaptive2,goldenhsluger2007structural,goldenshluger2008universalpointwise,goldenshluger2011bandwidthselection,goldenshluger2013bandwidthselection} is similar, except that the bias is replaced by a comparative bias (within the model class), which is in itself defined as the extremum of a penalized expression. For instance, it has been studied in the context of non-private projection estimators in \cite{comte2012adaptiveregression,chagny2013penalizationvslepski,bertin2016adaptive}. However, to the best of our knowledge, it has never been used as a privacy-aware selection mechanism in the context of central differential privacy before. A nice overview of non-private adaptive methods is presented in \cite{chagny:hal-02132884}.

In the literature of differential privacy, there are clever ways to perform model selection (which is here used as a synonym of adaptivity) without having to split the privacy budget (with composition theorems like \Cref{factCompositionConcentratedDifferentialPrivacy}) between all the models to choose from (e.g. the \emph{Exponential Mechanism} \cite{mcsherry2007mechanism}, \emph{Report Noisy Max} \cite{dwork2014algorithmic} or the \emph{Permute-And-Flip} mechanism \cite{mckenna2020permuteandflip,ding2021permuteandflip}). Such methods have found their way in multiple applications \cite{hardt2012privatedatarelease,blocki2016privatepasswordfrequency,smith2011privacy,bhaskar2010privatefrequentpattern,liu2019privateselection}. Unfortunately, the adaptive estimation procedure that we adopt here does not adequately fit in any of those frameworks, and we will thus resort to using composition theorems for the model selection. A blessing of the procedure that is presented here, however, is that it only needs to select between very few models (typically of the order of a polynomial of $\log (n)$), and the degradation of utility will hence be small.

\paragraph{Under local privacy.}
For completeness, we include references for related problems in the \emph{local} model of privacy (that we recall is different to the model of this article). In this setup, nonparametric density estimation was studied in \cite{duchi2013local,duchi2018minimax,butucea2020local,kroll2021density,schluttenhofer2022adaptive,gyorfi2023multivariate}. In \cite{butucea2020local}, adaptivity is obtained by leveraging the properties of the wavelet basis that is used for the estimation. \cite{kroll2021density} uses a variant of the Lepskii method for adaptivity, with the twist that the level of privacy is fixed beforehand. In \cite{schluttenhofer2022adaptive}, the authors modify the latter to be adaptive to the level of privacy as well. Our results differ from theirs by the model of privacy, and by the fact that they look at the estimation of the density at a single point whereas we look at the estimation of the density on the whole support. In particular, the rates of estimation are different. Finally, nonparametric regression was studied in \cite{berrett2021strongly,gyorfi2022rate},  nonparametric tests were studied in \cite{lam2022minimax}, and recently, nonparametric locally-private Bayesian modeling was proposed in \cite{beraha2023mcmc}.

\subsection{Contributions}

The main contributions of this article could be summarized as follows :

\paragraph{Adaptivity.} The main contribution of this article is to propose an adaptive estimator based on the Lepskii method that almost matches the performance of the optimal estimator, without prior knowledge of the smoothness of the density of interest. 

In practice, it means that the practitioner does not have to have strong prior information about the density to estimate in order to estimate it near-optimaly. 

Adaptivity is an important property in statistics and in particular with density estimation, and to the best of our knowledge, no concurrent work for density estimation in the context of central differential privacy has presented such adaptive procedure before. 

\paragraph{Non-integer levels of smoothness.} While the results of \cite{lalanne2023about} coincide with the ones presented in this article in the case of \emph{integer-valued} $\beta$'s (in dimension $1$), the authors did not mention the eventuality of more fine-grained levels of smoothness. 

This choice seems to be entirely motivated by technical reasons, and is unsatisfactory in practice. Indeed, it forces one to model the density of interest by a conservative smoothness level, which in turn leads to suboptimal convergence speeds. Real-values levels of smoothness allow for a much finer-grained modeling of the densities of interest.

A usual trick for generalizing consists in defining the class of densities of interest in terms of their Fourier coefficient instead of their derivatives (which was the reason for the integer-valued smoothness in the first place). However, such definition does not lead to provably good lower-bounds under differential privacy with the techniques presented in \cite{lalanne2023about}.

Instead, in this article, we circumvent that difficulty by considering an extended definition of Sobolev spaces via Hölder remainders (see \Cref{sec:realSobolevSpaces}). This definition is a bit harder to work with, yet it has the advantage of leading to tight lower and upper-bounds for any non-negative $\beta$. In particular, we believe that certain technical results developed here such as \Cref{prop:Fourier_Sobolev_estimates} may be of independent interest

\paragraph{Arbitrary dimension.} Finally, the last contribution of this article is to generalize previous results to an arbitrary dimension $d$. In particular, the effects of dimensionality on the estimation and on the privacy-utility tradeoff are discussed in \Cref{subsec:qualitativediscussion}. While the techniques for this generalization are rather straightforward, we believe that presenting results in this general form allow understanding the links between dimensionality and privacy, which is important in practice.

As a final note, we would like to highlight that all the proposed methods are of polynomial complexities.

\subsection{Notations}

$\N$, $\Z$, $\R$ and $\C$ are respectively used to refer to the sets of natural numbers (including $0$), relative numbers, real numbers, and complex numbers. In order to avoid confusion with indexes, we note $i_{\C}$ the canonical complex square root of $-1$. If $x \in \C$, $\bar{x}$ is used to refer to its conjugate complex number, $|z|$ to its modulus, $R(z)$ to its real part and $I(z)$ to its imaginary part. We equip $\C^d$ with its standard Hermitian product $\langle \cdot, \cdot \rangle$, and its associated norm is noted $\|\cdot\|$. We note $B(x, r)$ the open ball or radius $r$ centered in $x$ for $\|\cdot\|$. For $p \in \N \setminus \{ 0\} \cup \{ + \infty\}$, $\|\cdot\|_p$ refers to the usual $l_p$ norm for complex-valued vectors (in particular $\|\cdot\| = \|\cdot\|_2$), and to the usual $L^p$ norm for complex-valued measurable functions. For any $k \in \N$, $\mathcal{C}^k(\mathcal{S})$ is used to refer to the set of functions from a space $\mathcal{S}$ to $\C$ that are $k$ times continuously differentiable. $\mathcal{C}^{\infty}(\mathcal{S})$ is used to refer to $\cap_{k \in \N} \mathcal{C}^k(\mathcal{S})$. For a multi-index $a = (a^1, \dots, a^d) \in \N^d$, $|a|$ is used to refer to the \emph{length} of $a$, which is $\sum_{i=1}^d a^i$. 

For a multi-index $a = (a^1, \dots, a^d) \in \N^d$ and $b \in \C$, we define $b a \eqdef (b a^1, \dots, b a^d)$, $b^a \eqdef b^{|a|}$, and $a^b \eqdef ((a^1)^b, \dots, (a^d)^b)$. Furthermore, if $b = (b^1, \dots, b^d) \in \C^d$, $b^{\times a} \eqdef (b^1)^{a^1} \times \dots \times (b^d)^{a^d}$.
In context, this conflict small conflict of notations shouldn't be an issue.
Given a $k \in \N$, $f \in \mathcal{C}^k(\R^d)$, and a multi-index $a = (a^1, \dots, a^d) \in \N^d$ such that $|a| \leq k$, we use the notation
\begin{equation*}
    \partial^a f \eqdef \frac{\partial^{|a|} f}{\partial_1^{a^1} \partial_2^{a^2} \dots \partial_d^{a^d}} \;,
\end{equation*}
where $\partial / \partial_i$ is used to refer to the derivation w.r.t. the $i^{th}$ component in the canonical basis of $\R^d$. Alternatively, we may also note $f^{(a)}$ as a short for $\partial^a f$. $\mathcal{N}(\mu, \Sigma)$ refers to the multivariate normal distribution of mean vector $\mu$ and of covariance matrix $\Sigma$. When a distribution is used in vector calculus (e.g. $a + \mathcal{N}(\mu, \Sigma)$), the distribution has to be understood as a random variable with the desired distribution. Without further specification, it is taken independent of the rest of the stochastic quantities of the article. For a density of probability $f$, we may simply refer by $f$ the probability distribution associated with it. The rest of the notations are introduced within the article directly. 

\section{Differential privacy}

\label{sectionDifferentialPrivacy}

This section presents some basic background on differential privacy that will be needed for the rest of the article.

Given two datasets $\vect{X} = (X_1, \dots, X_n)\in \set{X}^n$ and $\vect{Y}= (Y_1, \dots, Y_n) \in \set{X}^n$ where $\set{X}$ is the feature space ($[0, 1]^d$ in this article), the \emph{Hamming} distance between $\vect{X}$ and $\vect{Y}$ is defined as
\begin{equation*}
    {\ham{\vect{X}}{\vect{Y}} \eqdef \sum_{i=1}^n \Ind_{X_i \neq Y_i}} \;.
\end{equation*}

\begin{definition}[{$\rho$-zCDP \cite{dwork2016concentrated,bun2016concentrated}}]
\label{definitionConcentratedDifferentialPrivacy}
    Given an output space $\set{O}$ and $\rho \in (0, +\infty)$, a randomized mechanism (i.e. a conditional kernel of probabilities) $\mech{M} : \set{X}^n \rightarrow \set{O}$ is $\rho$-zero concentrated differentially private ($\rho$-zCDP) if $\forall \vect{X}, \vect{Y} \in \set{X}^n$, $\ham{\vect{X}}{\vect{Y}} \leq 1 \implies $
\begin{equation*}
\label{eq:defcdp}
    \forall 1 < \alpha < +\infty\, : \renyi{\alpha}{\mech{M}(\vect{X})}{\mech{M}(\vect{Y})} \leq \rho \alpha ,
\end{equation*}
where $\renyi{\alpha}{\cdot}{\cdot}$ denotes the Renyi divergence of level $\alpha$, defined when $\alpha>1$ as: 
\begin{equation*}
    \renyi{\alpha}{\mathbb{P}}{\mathbb{Q}} \eqdef \frac{1}{\alpha - 1} \log \int  \p{\frac{d \mathbb{P}}{ d \mathbb{Q}}}^{\alpha - 1} d \mathbb{Q} \;.
\end{equation*}
For more details on this measure of divergence, please refer to \cite{van2014renyi}.
\end{definition}

\begin{lemma}[Privacy of the Gaussian mechanism (Proposition 6 with Lemma 7 in \cite{bun2016concentrated})]
\label{factProvacyGaussianMechanism}
    Given a deterministic function $h$ mapping a dataset to a quantity in $\R^{d'}$, one can define the $l_2$-sensitivity of $h$ as
    \begin{equation*}
\begin{aligned}
    \Delta_2 h \eqdef \sup_{\vect{X}, \vect{Y} \in \set{X}^n : \ham{\vect{X}}{\vect{Y}} \leq 1} \big\| h (\vect{X})  - h (\vect{Y})\big\|_2 \;.
\end{aligned}
\end{equation*}
When this quantity is finite, for any $\rho > 0$, the Gaussian mechanism defined as 
\begin{equation*}
    \begin{aligned}
        \vect{X} \mapsto h(\vect{X}) + \frac{\Delta_2 h}{\sqrt{2 \rho}} \mathcal{N}(0, I_{d'}) \;,
    \end{aligned}
\end{equation*}
is $\rho$-zCDP.
\end{lemma}

\begin{lemma}[{Adaptive composition of private mechanisms (Lemma 7 in \cite{bun2016concentrated})}]
\label{factCompositionConcentratedDifferentialPrivacy}
    If the private mechanisms $M_1(\cdot), M_2(\cdot , z)$ are respectively $\rho_1$-zCDP and $\rho_2$-zCDP for any context $z$, then the private mechanism $M_2(\cdot, M_1(\cdot))$ is $(\rho_1+ \rho_2)$-zCDP.
\end{lemma}
The last result can easily be generalized to a finite family of mechanisms by induction.

Finally, the last property of private mechanisms that we will use implicitly throughout this article is the data-processing inequality (or post-processing lemma in the language of differential privacy (Lemma 8 in \cite{bun2016concentrated})), which states that if $\mech{M}$ satisfies $\rho$-zCDP, then for any conditional kernel of probabilities $g$, $g \circ \mech{M}$ also satisfies $\rho$-zCDP. 

\section{(Private) projection estimators}
\label{sec:projectionestimators}

In Statistics, when the quantity to estimate $f$ belongs to some Hilbert space that admits a countable Hilbert basis $(\phi_k)_k$, projection estimators \cite{tsybakov2003introduction} usually refer to estimators of the form
\begin{equation*}
    \hat{f} = \sum_k \hat{\theta}_k \phi_k \;,
\end{equation*}
where the sum is usually truncated with a spectral cut-off of frequencies, and where $(\hat{\theta}_k)_k$ is a sequence of estimators of the true coefficients of the decomposition in the Hilbert basis. The name comes from the fact that such estimator mimics the orthogonal projection of $f$ onto the space spanned by the first vectors of this Hilbert basis. To the best of our knowledge, their first appearance in the context of differential privacy is in \cite{wasserman2010statistical}.

\subsection{Explicit construction}

We detail in \Cref{sectionSobolevSpacesUpperBounds} the exact functional spaces in which we assume the unknown density $f$ to be. For now, we only need to know that $f$ is in $L^2([0, 1]^d)$ equipped with Lebesgue's measure and its standard Hermitian product 
\begin{equation*}
    \langle f, g \rangle \eqdef \int_{[0, 1]^d}f \bar{g} \;,
\end{equation*} 
and its standard inherited norm $\| \cdot \|$.
We further fix the Hilbert basis $(\phi_k)_k$ of $L^2([0, 1]^d)$  as the one associated to the following Fourier basis :
\begin{equation}
\label{equationFourierBasisDefinition}
\begin{aligned}
    \forall k \in \Z^d, \quad \phi_k(x) &\eqdef e^{i_{\C} 2 \pi \langle k, x \rangle} \\
    &= e^{i_{\C} 2 \pi ( k_1 x^1 + \dots + k_d x^d )} \;.
\end{aligned}
\end{equation}
We also define $S_k \eqdef \vspan \p{\phi_k}_{k \in \{-M, \dots, M \}^d}$ the finite-dimensional vector space spanned by the $\phi_k$'s with every index in $k$ lower than $M$, and we define $f_M$ as the orthogonal projection of $f$ onto $S_M$.

From this, we define the natural (by the law of large numbers) estimators of the coefficients in the Fourier basis
\begin{equation}
\label{equationFourierEstimatorsDefinition}
    \tilde{\theta}_k \eqdef \frac{1}{n} \sum_{j=1}^n \bar{\phi}_k(X_j) 
     = \frac{1}{n} \sum_{j=1}^n e^{- i_{\C} 2 \pi ( k_1 X_j^d + \dots + k_d X_j^d )}
    \;,
\end{equation}
and their noisy estimates
\begin{equation}
\label{eq:noise_hat_coeff_f}
    \hat{\theta}_k \eqdef \tilde{\theta}_k + \sigma_M \xi_k \;,
\end{equation}
where $\sigma_M$ will be a variance factor that will be tuned later on to obtain the desired level of privacy, 
and $(\xi_k)_{k \in \mathbb{Z}^d}$ is an \textit{i.i.d.} complex Gaussian noise
\begin{equation}
    \label{def:noise_xi_k}
    \xi_k \sim \p{\mathcal{N}(0, 1) + i_{\C} \mathcal{N}(0, 1)} \;.
\end{equation}
Finally, we define the projection estimator at rank $M$ as 
\begin{equation}
\label{nonPrivateEstimatorDefinition}
    \Tilde{f}_M \eqdef \sum_{k \in \{-M, \dots, M \}^d} \Tilde{\theta}_k \phi_k \;,
\end{equation}
and its private counterpart as
\begin{equation}
\label{privateEstimatorDefinition}
    \hat{f}_M \eqdef \sum_{k \in \{-M, \dots, M \}^d} \hat{\theta}_k \phi_k \;.
\end{equation}

\subsection{General utility}

The general utility of the previous estimator is given by the following result : 

\begin{lemma}[General bias-variance decomposition of $\hat{f}_M$]
\label{lemmaBiasVarianceDecomposition}
    For any $M$, the estimator $\hat{f}_M$ satisfies 
    \begin{equation*}
        \begin{aligned}
            \E \p{\| f - \hat{f}_M\|^2} 
            \leq 
            &\underbrace{\| f - f_M\|^2}_{\blue{\text{Squared Bias}}} + 
            \underbrace{\frac{ (2M + 1)^d}{n}}_{\green{\text{Sampling Variance UB}}}\\
            &+ \underbrace{2 (2M + 1)^d \sigma_M^2}_{\red{\text{Privacy Noise Variance}}} \;.
        \end{aligned}
    \end{equation*}
\end{lemma}
\begin{proof}
    See \Cref{proofOfLemmaBiasVarianceDecomposition}.
\end{proof}

The bias term $\| f - f_M\|$ simply characterizes how well $f$ is approximated  in $S_M$. Controlling this term requires regularity assumptions on $f$, which is done in \Cref{sectionSobolevSpacesUpperBounds}.

\subsection{Privacy guarantees}

The privacy of this estimation procedure is given by the following theorem :
\begin{theorem}[Privacy of $\hat{f}_M$]
\label{theoremPrivacyProjectionEstimator}
    For any $M$, the mechanism $(X_1, \dots, X_n) \mapsto \hat{f}_M$ (or equivalently the mechanism that releases the computed $\hat{\theta}_k$'s for $k \in \{-M, \dots, M \}^d$) is $\rho$-zCDP if $\sigma_M = \frac{2 \sqrt{ (2M+1)^d}}{n\sqrt{\rho}}$.
\end{theorem}
\begin{proof}
    See \Cref{proofOfTheoremPrivacyProjectionEstimator}.
\end{proof}
If follows from the application of the classical privacy guarantees of the Gaussian mechanism.

\section{Upper-Bounds for different smoothness levels}
\label{sectionSobolevSpacesUpperBounds}

As explained in the last section, controlling the bias term $\| f - f_M\|$ requires regularity assumptions on $f$. This section solves this issue by imposing Sobolev-smoothness.

\subsection{Sobolev spaces in high dimension}

In order to simplify the reading flow of the article, its main body only presents spaces of \emph{integer} smoothness $\beta \in \N \setminus \{ 0\}$. All the results can be generalized to spaces of \emph{real} smoothness $\beta > 0$. With every result that we present for an integer $\beta$ in the main body of the article, we will talk about its counterpart in the case of real $\beta$, and we will link to the technical details in the appendix.

For $\beta \in \N \setminus \{0\}$ and $L>0$, the isotropic Sobolev space $\mathcal{S}_L(\beta)$ is defined as the subset of $\mathcal{C}^k([0, 1]^d)$ of functions of which the energy of the $\beta^\text{th}$ derivative is bounded by $L^2$. Namely, $f \in \mathcal{S}_L(\beta)$ if $f \in \mathcal{C}^k([0, 1]^d)$ and if 
\begin{equation*}
\label{equationBoundedEnergySobolev}
    \sum_{\alpha \in \N^d : |\alpha| = \beta} \int_{[0, 1]^d} |\partial^\alpha f|^2 \leq L^2\;.
\end{equation*}
$\beta$ is referred to as the smoothness parameter of the functional space $\mathcal{S}_L(\beta)$. For real $\beta$'s, Sobolev spaces are defined similarly, except that non-integer derivatives are handled via Hölder remainders (see \Cref{sec:realSobolevSpaces}).

As it is often the case when dealing with Fourier coefficients, it is convenient to define the \emph{periodic} Sobolev space $\mathcal{S}_L^p(\beta)$ by making sure that the functions and their derivatives are compatible with the typical periodicity of the Fourier basis. A function $f \in \mathcal{S}_L(\beta)$ is in $\mathcal{S}_L^p(\beta)$ if for any multi-index $\alpha \in \N^d$ of length at most $\beta$ (strict) and any $x = (x_1, \dots, x_d) \in [0, 1]^d$, $x_i \in \{0, 1\} \implies$
\begin{equation}
\label{equationPeriodicityCondition}
    \partial^\alpha f (x) = \partial^\alpha f (x_1, x_{i-1}, 1- x_{i}, x_{i+1}, \dots, x_d) \;.
\end{equation}
The definition of periodic spaces is identical in the case of real-valued $\beta$'s.

\subsection{Implications on the bias}

The Sobolev-smoothness of $f$ imposes that its Fourier coefficient have a polynomial decrease (see \Cref{lemmaFourierShrinkage}). This property may in turn be used to control the bias of with the following lemma :

\begin{lemma}[Bias of $\hat{f}_M$ with Sobolev assumption]
\label{lemmaBiasUBSobolev}
    For any $M$, if $f \in \mathcal{S}_L^p(\beta)$, then the bias of $f_M$ satisfies
    \begin{equation*}
        \begin{aligned}
            \| f - f_M\|^2
            \leq 
            &\frac{L^2}{(2 \pi)^{2 \beta}} \frac{1}{(M+1)^{2 \beta}}\;.
        \end{aligned}
    \end{equation*}
\end{lemma}
\begin{proof}
     See \Cref{proofOfLemmaBiasUBSobolev}.
\end{proof}

In the case of real-valued $\beta$'s, a similar control on the bias is given in \Cref{prop:Fourier_Sobolev_estimates}. Its main conceptual difference with \Cref{lemmaBiasUBSobolev} is that it adds a linear dependence in the dimension.

\subsection{Estimation upper-bound in Sobolev spaces}

Combining \Cref{lemmaBiasUBSobolev} and \Cref{lemmaBiasVarianceDecomposition}, and then optimizing over $M$ yields the following upper-bound for the private statistical estimation in $\mathcal{S}_L^p(\beta)$ :

\begin{theorem}[Upper-bound in $\mathcal{S}_L^p(\beta)$]
\label{theoremUpperBoundSobolevClass}
    There exists a positive $C$ that depends on $\beta$ and $L$ only such that, if $f \in \mathcal{S}_L^p(\beta)$, and if the values $M$ and $\sigma_M$ are tuned as
    \begin{equation*}
        M + 1 = \min \left\{ \left\lfloor \p{n / 2^d}^{\frac{1}{2 \beta + d}} \right\rfloor, \left\lfloor\p{n \sqrt{\rho} / 2^d}^{\frac{1}{\beta + d}} \right\rfloor \right\} \;,
    \end{equation*} 
    and $\sigma_M = \frac{2 \sqrt{(2M+1)^d}}{n\sqrt{\rho}}$, then the mechanism that returns $\hat{f}_M$ is $\rho$-zCDP and its error is bounded as
    \begin{equation*}
        \E \p{\| f - \hat{f}_M\|^2} \leq C  (M+1)^{-2 \beta}.
    \end{equation*}
\end{theorem}
\begin{proof}
    See \Cref{proofOfTheoremUpperBoundSobolevClass}.
\end{proof}
\Cref{lemmaBiasUBSobolev} and \Cref{prop:Fourier_Sobolev_estimates} are similar enough that the only adaptation to \Cref{theoremUpperBoundSobolevClass} needed to make it work for integer-valued $\beta$'s is to add that $C$ also depends linearly on $d$. In particular, the scaling in $n$ and $\rho$ remains the same.

\section{Lower-bounds and minimax optimality}
\label{sec:lowerbounds}

This section presents lower-bounds on the private estimation in $\mathcal{S}_L^p(\beta)$, and discusses on the role of the different parameters on the difficulty of estimation.

\subsection{Quantitative lower-bound}

We have the following lower-bound, which generalizes the results of \cite{lalanne2023about} in general dimension $d$ :
\begin{theorem}[Lower-bound in $\mathcal{S}_L^p(\beta)$]
\label{theoremLowerBoundSobolev}
    There exist two positive constants $C_1$ and $C_2$ depending on $L$, $\beta$ and $d$ only such that, for any $n$ and $\rho$, if $\hat{f}$ satisfies $\rho$-zCDP, then there exists $f \in \mathcal{S}_L^p(\beta)$ such that
    \begin{equation*}
        \E_{f} \p{\| f - \hat{f}\|^2} 
        \geq
        C_1 
         \max \left\{
        n^{-\frac{2 \beta}{2 \beta  + d}}, (n \sqrt{\rho})^{- \frac{2 \beta}{\beta+d}})
        \right\}
    \end{equation*}
    as soon as $\min \left\{ n, n\sqrt{\rho} \right\} \geq C_2$.
\end{theorem}
\begin{proof}
    See \Cref{proofOfTheoremLowerBoundSobolev}.
\end{proof}
For real-valued $\beta$'s, this result also holds. \Cref{sec:adaptation_proof_lower_bound} discusses the adaptation of the proof of \Cref{theoremLowerBoundSobolev} to this more general case.

\Cref{theoremLowerBoundSobolev}, when compared to the upper-bound given in \Cref{theoremUpperBoundSobolevClass} allows concluding that private projection estimators converge at the minimax-optimal rate 
\begin{equation}
    \label{def:rate}
    r_{n,\rho}(\beta) \eqdef  \max \left\{
        n^{-\frac{2 \beta}{2 \beta  + d}}, (n \sqrt{\rho})^{- \frac{2 \beta}{\beta+d}})
        \right\}, \quad \forall \beta >0 \;,
\end{equation}
up to a multiplicative constant depending on $L$, $\beta$ and $d$ only. While the dependence in those quantities is easily explained in the upper-bounds, a caveat of the proof of \Cref{theoremLowerBoundSobolev} is that the dependence is \emph{implicit} by construction, and that no closed-form formula may easily be obtained. 

\subsection{Qualitative implications}
\label{subsec:qualitativediscussion}

From this optimal rate of estimation, we may describe the effects of the different parameters of the privacy-utility tradeoff.

\begin{itemize}
    \item The privacy parameter $\rho$ :
    The two important regimes of estimation are $\rho \gtrsim n^{-\frac{2 \beta}{2 \beta + d}}$ where $\gtrsim$ should be understood as "greater up to a multiplicative constant" and its complement $\rho \ll n^{-\frac{2 \beta}{2 \beta + d}}$. In the first regime, when the level of privacy is not too high compared to the amount of data, privacy comes at a negligible cost on the estimation. On the other hand, in the complementary regime, the utility can be arbitrarily degraded by making $\rho$ arbitrarily small.
    \item The smoothness $\beta$ :
    The higher $\beta$, the smaller the cut-off rate $n^{-\frac{2 \beta}{2 \beta + d}}$. In other words, the smoother the density to estimate, the more private the estimation can be with no significant degradation of utility.
    \item The dimensionality $d$ : 
    Dimensionality has the converse effect on the cut-off rate. The higher the dimension, the more data will be needed to make the effects of privacy negligible. Furthermore, the cut-off itself is affected by the curse of dimensionality.
\end{itemize}

\section{Adaptivity}
\label{sectionAdaptivity}

As seen previously, it is possible to design a private mechanism via projection estimators that is minimax optimal for the class of densities in $\mathcal{S}_L^p(\beta)$ in dimension $d$. 

However, to do so, the optimal cut-off frequency: 
\begin{equation*}
        M_{n,\rho}(\beta) \eqdef \min \left\{ \left\lfloor n ^{\frac{1}{2 \beta + d}} \right\rfloor, \left\lfloor\p{n \sqrt{\rho}}^{\frac{1}{\beta + d}} \right\rfloor \right\} \;
    \end{equation*} 
is chosen based on the knowledge on $n$, $\rho$, $d$ and $\beta$ (see \Cref{theoremUpperBoundSobolevClass}). For the practitioner, the knowledge of $n$, $\rho$ and $d$ is not difficult. The knowledge of $\beta$ on the other hand is a much stronger hypothesis, and it already implies a strong prior knowledge on $f$. This section presents a private estimation strategy that is \emph{adaptive} in the sense that it does not require the prior knowledge of $\beta$, while almost achieving the utility of \Cref{theoremUpperBoundSobolevClass} (up to polylogarithmic factors and negligible terms).

\subsection{A first candidate for private selection}

At first, an idea for private adaptive estimation could be to :
\begin{enumerate}[(i)]
    \item Compute a non-private adaptive estimator of the density with classical methods (like for instance the non-private Lepskii method \cite{lepskii1991adaptivegaussiannoise}).
    \item Then add noise to its Fourier coefficients in order to make it private.
\end{enumerate}
However, there is a trap with this method that one must not fall into : the adaptive truncation rank $\hat{M}$ that is selected by the non-private adaptive method is a quantity that is built from the data, and it may leak user's information. It is thus not possible to simply add noise to the Fourier coefficients of the non-private Fourier coefficients up to truncation rank $\hat{M}$ with magnitude $\sigma_{\hat{M}}$ calibrated as is \Cref{theoremPrivacyProjectionEstimator} and to call the result differentially private. Instead, one must add noise to the Fourier coefficients up to a truncation rank that is either fixed in advance, or that builds on the data, in which case the privacy budget of such will have to be accounted for. The problem with such method is that classical adaptive methods will only try to balance the bias and the sampling variance, but won't account for the privacy variance. In particular, when $\rho$ is small, it is unclear if this method may have the optimal rate of convergence. In the next subsection, we detail the alternative method that we chose, that balances the bias, the sampling variance and the privacy variance \emph{at the same time}, leading to private and adaptive near-optimal estimation.

\subsection{Private and privacy-aware  Lepskii method \label{sec:privacy-adaptive}}

Multiple flavors of the Lepskii method exist in the literature. Here, we present our adaptations of the main two ones to the context of private model selection. We discuss the advantages and the drawbacks of each method.

\subsubsection{Risk penalization}

We introduce the penalized risk (up to a useful log term):
\begin{equation}
    \label{def:penalty_risk}
     r_{n,\rho}(\beta)^* \eqdef  C (\log n)^a r_{n,\rho}(\beta),
\end{equation}
where $C>1$ and $a>0$ are some constants independent from $n$ and $\rho$ that will be specified later on.
We introduce a grid $\mathbb{B}$ on the possible values of $\beta$ that ranges between $0$ and $\log n$, defined by:
\begin{equation}
    \label{def:beta_grid}
    \begin{aligned}
    \mathbb{B}_n := \bigg\{ &\beta_0=\frac{k_n \epsilon}{\log(n)}, \beta_1 = \beta_0 - \frac{\varepsilon}{\log(n)}, \\
    &\quad\quad
    \beta_2 = \beta_1 - \frac{\varepsilon}{\log(n)} , \ldots, \beta_{k_n-1} \ge 0 \bigg\} \;.
    \end{aligned}
\end{equation}
The number of possible values for $\beta$ in $\mathbb{B}_n$ is then denoted by $k_n$ and $k_n = \lfloor \varepsilon^{-1} \log^2 n \rfloor$.

Our Lepskii decision rule is built upon the estimation of the smoothness parameter with the computation of a collection of estimators for several values of $\beta \in \mathbb{B}_n$ and then with a clever selection among theses values with the help of a trade-off criterion. Thanks to \Cref{factCompositionConcentratedDifferentialPrivacy}, to ensure a desired level of privacy of our final estimator, we introduce 
\begin{equation}
\label{eq:rho'}
    \rho'_n = \rho \varepsilon \log^{-2} n.
\end{equation}
We are ready to define our \textit{adaptive} selection rule as:
\begin{equation}
    \label{def:lepski}
    \begin{aligned}
    \hat{m}_n := \inf \bigg\{ &m \ge 0 \, : \quad  \forall \ell \ge m, \\
    &\|\hat{f}_{M_{n,\rho'_n}(\beta_m)}-\hat{f}_{M_{n,\rho'_n}(\beta_\ell)}\|_2^2 \leq 
    r_{n,\rho'_n}(\beta_{\ell})^*
    \bigg\}
    \end{aligned}
\end{equation}

For the sake of clarity, we will use the following shortcut of notations to improve the readability of our paper:
$$
  \hat{f}_{\hat{M}} \eqdef \hat{f}_{M_{n,\rho'_n}(\beta_{\hat{m}_n})} \quad \text{and} \quad 
\hat{M} \eqdef M_{n,\rho'_n}(\beta_{\hat{m}_n}) \;.
$$

We establish the following result.
\begin{theorem}
    \label{theo:Lepskii_rate}
Assume that $a\ge 1$, $C \ge 8L^2 \vee 2^{2d+9}$ and $n\ge 3$, $\rho \leq 1$, if $\hat{f}_{\hat{M}}$ is the adaptive estimator selected with the Lepskii rule, then $\hat{f}_{\hat{M}}$ is $\rho$-zCDP and it satisfies the risk upper bound
    \begin{equation*}
        \begin{aligned}
\mathbb{E}&\p{\|\hat{f}_{\hat{M}} -f\|_2} \leq 2 \sqrt{   r_{n,\rho_n'}(\beta)^*  } 
  \exp\left( \frac{\varepsilon}{\beta+ d}\right) \\&+ 
  \sqrt{8 (2+d) } \varepsilon^{-3/2}   \left( 1+ {\rho'_n}^{-\frac{ 1 }{2(1  + d)}}\right) \log^2 n n^{-2}
  \end{aligned}
\end{equation*}
\end{theorem}
\begin{proof}
    See \Cref{proofoftheo:Lepskii_rate}
\end{proof}

\paragraph{Comments.} This result shows that out of the box (i.e. without additional assumptions on $f$), $\hat{f}_{\hat{M}}$ nearly matches the optimal speed of estimation up to negligible terms, and by excluding the fact that we did not take the error squared, but simply the error in $L^2$ distance.

\subsubsection{Penalization of the estimated bias}

Let $\mathcal{M}$ be the collection of spectral cut-offs. The following method describes how to choose $\hat{M}$, and the associated $\hat{f}_{\hat{M}}$. We start by estimating $\tilde{f}_M$ and $\hat{f}_M$ for any $M \in \mathcal{M}$ with $\sigma_M = \frac{\sqrt{2 (2M+1)^d}}{n\sqrt{\rho'}}$ where $\rho'$ is tuned to obtain $\rho$-zCDP in the end as $\rho' = \frac{\rho}{|\mathcal{M}|}$. 

Then for any $M$, we define the following estimator of the squared bias of $\hat{f}_M$:
\begin{equation}
    B^2(M) \eqdef \max_{M' \in \mathcal{K}} \left\{ \|\proj_{S_{M'}}(\hat{f}_M) - \hat{f}_{M'} \|^2 - \Lambda^{(1)}(M') \right\}
\end{equation}
where $\Lambda^{(1)}(\cdot)$ is a penalization term that is fixed later. 
Then, $\hat{M}$ is chosen as the minimizer of the penalized estimated squared bias.
\begin{equation}
    \hat{M} \eqdef \argmin_{M \in \mathcal{M}} \left\{ B^2(M) +  \Lambda^{(2)}(M)\right\} \;.
\end{equation}
Again, $\Lambda^{(2)}(\cdot)$ is a penalization term that is fixed later on.

$\hat{f}_{\hat{M}}$ satisfies the following oracle inequality :
\begin{theorem}[]
\label{theoremAdaptiveSelectionUtility}
When computed with $\sigma_M = \frac{2 \sqrt{(2M+1)^d}}{n\sqrt{\rho / |\mathcal{M}|}}$ for any $M \in \mathcal{M}$, the mechanism that releases $\hat{f}_{\hat{M}}$ satisfies $\rho$-zCDP. Furthermore, there exist two absolute constants $C_1>0$ and $C_3>0$ and a quantity $C_2 > 0$ depending only on $\|f\|_{\infty}$\footnote{This dependence arises because of a technical argument in the proof.} such that, if for any $M$, 
\begin{equation*}
    \Lambda^{(1)}(M) =  \frac{96 (2M+1)^d}{n} + \frac{96 (2M+1)^{2d}}{n^2 \rho / |\mathcal{M}|}
\end{equation*}
and 
\begin{equation*}
    \Lambda^{(2)}(M) = \Lambda^{(1)}(M) + \frac{16 (2M+1)^{2d}}{n^2 \rho / |\mathcal{M}|} \;,
\end{equation*}
and if $(2 \max \mathcal{M} + 1)^d \leq n$,
then $\hat{f}_{\hat{M}}$ satisfies 
\begin{equation}
\begin{aligned}
    \E \bigg(&\|f - \hat{f}_{\hat{M}} \|^2\bigg)\leq \\
    &
    C_1 \min_{M \in \mathcal{M}}
    \bigg\{
    \| f - f_M\|^2 
    + \frac{ (2M + 1)^d}{n} \\
    &\quad\quad\underbrace{\quad\quad\quad\quad\quad\quad\quad\quad\quad\quad+ 2 (2M + 1)^d \sigma_M^2
    \bigg\} }_{\text{\blue{Best bias-variance tradeoff in $\mathcal{M}$ with privacy budget $\frac{\rho}{|\mathcal{M}|}$}}}\\
    &\quad\quad+  \underbrace{\frac{C_2}{n}}_{\green{\text{Sampling residual}}}
    +  \underbrace{\frac{C_3 |\mathcal{M}|}{ n^2 \rho}}_{\red{\text{Privacy residual}}} \;.
\end{aligned}
\end{equation}
\end{theorem}
\begin{proof}
    See \Cref{proofOfTheoremAdaptiveSelectionUtility}.
\end{proof}
Since the bias is left uncontrolled in this result, it remains true in the case of real-valued $\beta$'s.

When the collection of spectral cut-offs $\mathcal{M}$ is adequately chosen, this oracle inequality may be used to prove near-optimal convergence speed.
\begin{theorem}
\label{theoremAdaptivePrivateUtility}
    There exist a $C_1 >0$ depending on $\beta$ and $L$, and a $C_2 > 0$ depending on $\beta$, $d$ and $\|f\|_{\infty}$ such that, if $\min \{n, n \sqrt{\rho / \log_2\p{n}} \} \geq C_2$, then $\hat{f}_{\hat{M}}$ computed with 
    \begin{equation*}
        \mathcal{M} = \left\{ 1, 2, 4, \dots, 2^{\left\lfloor \log_2\p{\frac{n^{1/d} - 1}{2}} \right\rfloor} \right\}
    \end{equation*}
    and all the other hyperparameters set as in \Cref{theoremAdaptiveSelectionUtility} is $\rho$-zCDP and its utility satisfies
    \begin{equation}
    \begin{aligned}
        \E \Bigg(\|f - &\hat{f}_{\hat{M}} \|^2\bigg)\leq \\
        &C_1 \max \Bigg\{ n^{-\frac{2 \beta}{2 \beta  + d}} , \p{\frac{n \sqrt{\rho}}{\sqrt{\log_2\p{n}}}}^{- \frac{2 \beta}{\beta+d}} \Bigg\} \;.
        \end{aligned}
    \end{equation}
\end{theorem}
\begin{proof}
    See \Cref{proofOfTheoremAdaptivePrivateUtility}.
\end{proof}
Because of the extra dimensionality term in the control of the bias in the case of real-valued $\beta$'s, this result remains true in this case if one adds that $C_1$ also depends on $d$.

\paragraph{Comments.} Contrary tho the last procedure, this new one is near-optimal in terms of \emph{squared} error, at the cost of the control of $\|f\|_{\infty}$. As explained before, this requirement comes from a technical detail in the proof, and it might be an artifact of a suboptimal analysis from us. Also, the polylogarithmic degradation only affects the privacy term.

\section*{Impact Statement}
This paper presents work whose goal is to advance the field of Machine Learning. There are many potential societal consequences of our work, none which we feel must be specifically highlighted here.

\section*{Acknowledgements}
Clément Lalanne acknowledges the help of Aurélien Garivier and Rémi Gribonval on previous work on which this article builds on. The authors would like to thank Jérôme Bolte for the interesting discussions regarding the geometrical interpretations of differential privacy. This article was funded by the ANR MaSDOL and the Centre Lagrange.



\bibliography{biblio}
\bibliographystyle{icml2024}

\newpage
\appendix
\onecolumn

\section{Proofs of \Cref{sec:projectionestimators}}

\subsection{Proof of \Cref{lemmaBiasVarianceDecomposition}}
\label{proofOfLemmaBiasVarianceDecomposition}

Our starting point is the Parseval equality, that leads to:
\begin{equation}
        \begin{aligned}
            \E \p{\| f - \hat{f}_M\|^2} 
            &\stackrel{\text{Parseval}}{=}
            \E \p{
            \sum_{k \in \Z^d \setminus \{-M, \dots, M\}^d} |\theta_k|^2
            +
            \sum_{k \in \{-M, \dots, M\}^d} |\theta_k - \hat{\theta}_k|^2} \;,
        \end{aligned}
    \end{equation}
where the family $(\theta_k)_k$ refers to the Fourier coefficients of $f$ in the basis defined in \eqref{equationFourierBasisDefinition}, and where the noisy Fourier coefficient estimators $\p{\hat{\theta}_k}_{k \in \{-M, \dots, M \}^d}$ are defined in \eqref{eq:noise_hat_coeff_f}.
First, we may notice that :
\begin{equation}
    \sum_{k \in \Z^d \setminus \{-M, \dots, M\}^d} |\theta_k|^2 = \| f - f_M\|^2
\end{equation}
deterministically. This leads to the bias term in the error decomposition.

Then, for any $k \in \{-M, \dots, M \}^d$,
\begin{equation}
    \begin{aligned}
        \E \p{|\theta_k - \hat{\theta}_k|^2}
        &\leq \left|\E \p{\hat{\theta}_k} - \theta_k\right|^2 + \V \p{\hat{\theta}_k} \;.
    \end{aligned}
\end{equation}
Furthermore, 
\begin{equation}
    \begin{aligned}
        \E \p{\hat{\theta}_k}
        &\stackrel{\eqref{eq:noise_hat_coeff_f}}{=}
        \E \p{\tilde{\theta}_k + \sigma_K \p{\mathcal{N}(0, 1) + i_{\C} \mathcal{N}(0, 1)}} 
         \stackrel{}{=}
        \E \p{\tilde{\theta}_k } 
        \stackrel{\eqref{equationFourierEstimatorsDefinition}}{=}
        \E \p{\frac{1}{n} \sum_{i=1}^n \bar{\phi}_k(X_i)} \\
        &\stackrel{}{=}
        \frac{1}{n} \sum_{i=1}^n \E \p{\bar{\phi}_k(X_i)} 
        \stackrel{}{=}
        \frac{1}{n} \sum_{i=1}^n \theta_k = \theta_k \;.
    \end{aligned}
\end{equation}
Finally,
\begin{equation}
    \begin{aligned}
        \V \p{\hat{\theta}_k}
        &\stackrel{\eqref{eq:noise_hat_coeff_f}}{=}
        \V \p{\tilde{\theta}_k + \sigma_K \p{\mathcal{N}(0, 1) + i_{\C} \mathcal{N}(0, 1)}}
        \stackrel{\text{Indep.}}{=}
        \V \p{\tilde{\theta}_k } +  \V \p{ \sigma_K \p{\mathcal{N}(0, 1) + i_{\C} \mathcal{N}(0, 1)}} \\
        &\stackrel{\eqref{equationFourierEstimatorsDefinition}}{=}
        \V \p{\frac{1}{n} \sum_{i=1}^n \bar{\phi}_k(X_i)} +  \V \p{\sigma_M \p{\mathcal{N}(0, 1) + i_{\C} \mathcal{N}(0, 1)}} 
        \stackrel{\text{Indep.}}{=}
        \frac{1}{n^2} \sum_{i=1}^n \V \p{\bar{\phi}_k(X_i)} +  2 \sigma_M^2 \\
        &\stackrel{|\phi_k(\cdot)| \leq 1 \text{ \& \Cref{lemmaProvicius}}}{\leq}
        \frac{1}{n^2} \sum_{i=1}^n 1 +  2 \sigma_M^2 = 
        \frac{1}{n} + 2 \sigma_M^2
    \end{aligned}
\end{equation}

\subsection{Proof of \Cref{theoremPrivacyProjectionEstimator}}
\label{proofOfTheoremPrivacyProjectionEstimator}

The mechanism $(X_1, \dots, X_n) \mapsto \hat{f}_M$ may equivalently be seen as the mechanism that releases the vector in $\C^{(2M+1)^d}$ of the privatized Fourier coefficient estimates, or as the mechanism that releases the vector in $\R^{2(2M+1)^d}$ of the real and imaginary parts (respectively noted $R(\cdot)$ and $I(\cdot)$) of the privatized Fourier coefficient estimates.

We aim to apply Lemma \ref{factProvacyGaussianMechanism}: for this purpose, consider any multi-index $k$, and any $(X_1, \dots, X_n), (X_1', \dots, X_n') \in [0, 1]^d$,
\begin{equation}
\begin{aligned}
    \left|\tilde{\theta}_k(X_1, \dots, X_n) - \tilde{\theta}_k(X_1', \dots, X_n')\right| 
    &= 
    \left|
    \frac{1}{n} \sum_{j=1}^n \bar{\phi}_k(X_j) 
    -
    \frac{1}{n} \sum_{j=1}^n \bar{\phi}_k(X_j') 
    \right| \\
    &\leq 
    \frac{1}{n} \sum_{j=1}^n \left| \bar{\phi}_k(X_j) - \bar{\phi}_k(X_j') \right| \\
    &\stackrel{|{\phi}_k(\cdot)| \leq 1}{\leq} \frac{2 \ham{(X_1, \dots, X_n)}{(X_1', \dots, X_n')}}{n} \;.
    \;,
\end{aligned}
\end{equation}

Hence, for any $k$, $\tilde{\theta}_k$ is of $l_2$ sensitivity $\frac{2}{n}$. Hence, for any $k$, $R(\tilde{\theta}_k)$ and $I(\tilde{\theta}_k)$ are both of sensitivity at most $\frac{2}{n}$ (because $R(\cdot)$ and $I(\cdot)$ are orthogonal projections and are hence contraction linear mappings).

The $l_2$ sensitivity of computing the $2(2M+1)^d$ approximate real and imaginary parts of the Fourier coefficients is thus $\frac{2}{n}\sqrt{2(2M+1)^d}$.
Then, the application of \Cref{factProvacyGaussianMechanism} guarantees that the mechanism that releases $\hat{f}_M$, when computed with $\sigma_M = \frac{2 \sqrt{(2M+1)^d}}{n \sqrt{\rho}}$ satisfies $\rho$-zCDP.

\section{Proofs of \Cref{sectionSobolevSpacesUpperBounds}}

\subsection{Proof of \Cref{lemmaBiasUBSobolev}}
\label{proofOfLemmaBiasUBSobolev}

We will need the following lemma :
\begin{lemma}[Fourier tail in Sobolev spaces]
\label{lemmaFourierShrinkage}
    If $f \in \mathcal{S}_L^p(\beta)$, then 
    \begin{equation}
        \sum_{k \in \Z^d} \p{\sum_{\alpha \in \N^d: |\alpha| = \beta} (2 \pi k)^{\times 2 \alpha}} |\theta_k|^2 \leq L^2 \;,
    \end{equation}
    where $(\theta_k)_{k \in \Z^d}$ are the Fourier coefficients of $f$ w.r.t. the basis $(\phi_k)_{k \in \Z^d}$.
\end{lemma}
\begin{proof}
Let $\alpha \in \N^d$ such that $| \alpha | = \beta$ and let $k \in \Z^d$. Let us look $\theta_k^{(\alpha)}$ at the $k^\text{th}$ Fourier coefficient of $\partial^\alpha f$. Since $\beta \geq 1$, there exists $i_0 \in \N$ such that $\alpha_i \geq 1$. We note $\alpha_{-i_0}$ the multi-index with the same values as $\alpha$ except for its $i_0^{\text{th}}$ coordinate which has been decremented by $1$. Furthermore, for any $x \in [0, 1]^d$, any $i \in \{1, \dots, d\}$, and any $y \in [0, 1]$ we note $x^{(i=y)}$ the vector with the same components as $x$ but with $y$ as its $i^{\text{th}}$ component.

We have that 
\begin{equation}
    \begin{aligned}
        \theta_k^{(\alpha)} 
        &= \int_{[0, 1]^d} \partial^\alpha f (x_1, \dots, x_d) \bar{\phi_k}(x_1, \dots, x_d) dx_1 \dots dx_d \\
        &\stackrel{\text{Fubini}}{=} 
        \int_{[0, 1]^{d-1}} \p{ \int_{[0, 1]} \partial^\alpha f (x_1, \dots, x_d) \bar{\phi_k}(x_1, \dots, x_d)  dx_{i_0}}dx_1 \dots dx_{i_0 - 1} dx_{i_0 + 1} \dots dx_d \\
        &\stackrel{\text{}}{=} 
        \int_{[0, 1]^{d-1}} \p{ \int_{[0, 1]} \partial^\alpha f (x_1, \dots, x_d) e^{- i_{\C} 2 \pi ( k_1 x_1 + \dots + k_d x_d )}  dx_{i_0}}dx_1 \dots dx_{i_0 - 1} dx_{i_0 + 1} \dots dx_d \\
        &\stackrel{\text{I.B.P.}}{=} 
        \int_{[0, 1]^{d-1}} \Bigg(\partial^{\alpha_{-i_0}} f (x^{(i_0=1)}) e^{- i_{\C} 2 \pi \langle k,  x^{(i_0=1)} \rangle} - \partial^{\alpha_{-i_0}} f (x^{(i_0=0)}) e^{- i_{\C} 2 \pi \langle k,  x^{(i_0=0)} \rangle} \\
        &\quad\quad\quad\quad+ i_{\C} 2 \pi k_{i_0} \int_{[0, 1]} \partial^{\alpha_{-i_0}} f (x_1, \dots, x_d) e^{- i_{\C} 2 \pi ( k_1 x_1 + \dots + k_d x_d )}  dx_{i_0} \Bigg) dx_1 \dots dx_{i_0 - 1} dx_{i_0 + 1} \dots dx_d \\
        &\stackrel{\text{}}{=} 
        i_{\C} 2 \pi k_{i_0} \theta_k^{(\alpha_{-i_0})} \;.
    \end{aligned}
\end{equation}

Thus, by induction, we get that 
\begin{equation}
\label{equationAhjokdiqjsdhfoiq}
    \theta_k^{(\alpha)} = (i_{\C} 2 \pi k)^{\times \alpha} \theta_k \;.
\end{equation}

Next, since it holds for any $k$, we may write
\begin{equation}
\label{equationThiuboiuhabzef}
    \begin{aligned}
        \int_{[0, 1]^d} |\partial^\alpha f|^2 
        &\stackrel{\text{Parseval}}{=}  \sum_{k \in \Z^d} \langle \theta_k^{(\alpha)} , \theta_k^{(\alpha)} \rangle \\
        &\stackrel{\text{\eqref{equationAhjokdiqjsdhfoiq}}}{=}
        \sum_{k \in \Z^d} (2 \pi k)^{\times 2 \alpha} \langle \theta_k, \theta_k \rangle \;.
    \end{aligned}
\end{equation}

Finally, since this holds for any $\alpha$, we may sum over $\alpha$ and use \eqref{equationBoundedEnergySobolev} to get that 
\begin{equation}
    \begin{aligned}
        L^2 
        &\geq 
        \sum_{\alpha \in \N^d : |\alpha| = \beta} \int_{[0, 1]^d} |\partial^\alpha f|^2 \\
        &\stackrel{\text{\eqref{equationThiuboiuhabzef}}}{=}
        \sum_{\alpha \in \N^d : |\alpha| = \beta} \sum_{k \in \Z^d} (2 \pi k)^{\times 2 \alpha} \langle \theta_k, \theta_k \rangle \\ 
        &= \sum_{k \in \Z^d} \p{\sum_{\alpha \in \N^d: |\alpha| = \beta} (2 \pi k)^{\times 2 \alpha}} |\theta_k|^2 \;.
    \end{aligned}
\end{equation}
\end{proof}

If $k = (k_1, \dots, k_d) \in \Z^d \setminus \{-M, \dots, M \}^d$, then there exists $i_0$ such that
\begin{equation}
    |k_{i_0}| \geq M + 1 \;.
\end{equation}
By considering the multi-index $\alpha_0$ composed with only $0$'s except at the index $i_0$ to which we assign the value $\beta$, we thus get 
\begin{equation}
    (2 \pi (M+1))^{2 \beta} \leq (2 \pi k)^{\times 2 \alpha_0}\;,
\end{equation}
which allows writing
\begin{equation}
\label{jhvbkajhzbdkfjhbqhsdf}
    (2 \pi (M+1))^{2 \beta} \leq \sum_{\alpha \in \N^d: |\alpha| = \beta} (2 \pi k)^{\times 2 \alpha}
\end{equation}
since $\alpha_0$ is part of the summation indexes of the right-hand side.

Combining the last inequality with \Cref{lemmaFourierShrinkage} yields
\begin{equation}
\begin{aligned}
\sum_{k \in \Z^d \setminus \{-M, \dots, M \}} (2 \pi (M+1))^{2 \beta} |\theta_k|^2
&\stackrel{\eqref{jhvbkajhzbdkfjhbqhsdf}}{\leq} \sum_{k \in \Z^d \setminus \{-M, \dots, M \}} \p{\sum_{\alpha \in \N^d: |\alpha| = \beta} (2 \pi k)^{\times 2 \alpha}} |\theta_k|^2\\
&\leq
\sum_{k \in \Z^d} \p{\sum_{\alpha \in \N^d: |\alpha| = \beta} (2 \pi k)^{\times 2 \alpha}} |\theta_k|^2 \\
&\stackrel{\Cref{lemmaFourierShrinkage}}{\leq} L^2 \;.
\end{aligned}
\end{equation}
Hence,
\begin{equation}
    \begin{aligned}
        \frac{L^2}{(2 \pi)^{2 \beta}} \frac{1}{(M+1)^{2 \beta}}
        &\geq \sum_{k \in \Z^d \setminus \{-M, \dots, M \}}  |\theta_k|^2 \\
        &\stackrel{\text{Parseval}}{=} \| f - f_M\|^2 \;.
    \end{aligned}
\end{equation}

\subsection{Proof of \Cref{theoremUpperBoundSobolevClass}}
\label{proofOfTheoremUpperBoundSobolevClass}

The privacy of this mechanism is a direct consequence of \Cref{theoremPrivacyProjectionEstimator}.
Below, $\square$ will refer to a constant that depends on $\beta$ and $L$, whose value may change from line to line, and that is independent from  $n$ and $\rho$.

By combining \Cref{lemmaBiasVarianceDecomposition} and \Cref{lemmaBiasUBSobolev} with the value of the variance factor $\sigma_M = \frac{2 \sqrt{(2M+1)^d}}{n\sqrt{\rho}}$, we get that:
\begin{equation}
\label{lkjsbckjvhbqksdfqsd}
\begin{aligned}
    \E \p{\| f - \hat{f}_M\|^2} 
    &\stackrel{\Cref{lemmaBiasVarianceDecomposition}}{\leq}
    \| f - f_M\|^2 + 
    \frac{ (2M + 1)^d}{n} + 
    2 (2M + 1)^d \p{\frac{2 \sqrt{(2M+1)^d}}{n\sqrt{\rho}}}^2 \\
    &\stackrel{\Cref{lemmaBiasUBSobolev}}{\leq}
    \frac{L^2}{(2 \pi)^{2 \beta}} \frac{1}{(M+1)^{2 \beta}} + 
    \frac{ (2M + 1)^d}{n} + 
    2 (2M + 1)^d \p{\frac{2 \sqrt{(2M+1)^d}}{n\sqrt{\rho}}}^2 \\
    &\stackrel{2K+1 \leq 2(K+1)}{\leq}
    \frac{L^2}{(2 \pi)^{2 \beta}} \frac{1}{(M+1)^{2 \beta}} + 
    \frac{ 2^d (M + 1)^d}{n} + 
    2^{d+1} (M + 1)^d \p{\frac{2 \sqrt{2 ^d (M+1)^d}}{n\sqrt{\rho}}}^2 \\
    &\stackrel{}{\leq} \square
    \p{
    \frac{1}{(M+1)^{2 \beta}} + 
    \frac{ 2^d (M + 1)^d}{n} +\frac{2^{2d} (M+1)^{2d}}{n^2 \rho}} \;,
\end{aligned}
    \end{equation} 
We then find the optimal trade-off for $M$ by separating the regimes where the variance is dominated by the sampling noise or by the privacy noise.

\begin{itemize}
    \item \textbf{Bias - Sampling variance equilibrium :}
    We may first observe that
    \begin{equation}
        \frac{1}{(M+1)^{2 \beta}} \geq
    \frac{ 2^d (M + 1)^d}{n}
 \Longleftrightarrow      M+1 \leq \p{n / 2^d}^{\frac{1}{2 \beta + d}} \;.
    \end{equation}
    \item \textbf{Bias - Privacy variance equilibrium :}
    In the meantime, we get:
    \begin{equation}
        \frac{1}{(M+1)^{2 \beta}} \geq
    \frac{2^{2d} (M+1)^{2d}}{n^2 \rho}
    \Longleftrightarrow  M+1 \leq \p{n \sqrt{\rho} / 2^d}^{\frac{1}{\beta + d}} \;.
    \end{equation}
\end{itemize}

Hence, by taking 
\begin{equation}
        M + 1 = \min \left\{ \left\lfloor \p{n / 2^d}^{\frac{1}{2 \beta + d}} \right\rfloor, \left\lfloor\p{n \sqrt{\rho} / 2^d}^{\frac{1}{\beta + d}} \right\rfloor \right\} \;,
    \end{equation} 
we have that 
\begin{equation}
    \max \left\{ \frac{ 2^d (M + 1)^d}{n},  \frac{2^{2d} (M+1)^{2d}}{n^2 \rho}\right\}
    \leq \frac{1}{(M+1)^{2 \beta}} \;,
\end{equation}
and \Cref{lkjsbckjvhbqksdfqsd} yields
\begin{equation}
        \E \p{\| f - \hat{f}_M\|^2} \leq C \frac{1}{(M+1)^{2 \beta}} \;.
    \end{equation}

\section{Proofs of \Cref{sec:lowerbounds}}

\subsection{Proof of \Cref{theoremLowerBoundSobolev}}
\label{proofOfTheoremLowerBoundSobolev}

Let $m$ be an integer that will be specified later on in the proof. We consider the grid 
\begin{equation}
    \underbrace{\left\{ \frac{1}{m+1}, \frac{2}{m+1}, \dots, \frac{m}{m+1} \right\} \times \dots \times \left\{ \frac{1}{m+1}, \frac{2}{m+1}, \dots, \frac{m}{m+1} \right\}}_{d \text{ times}}\;.
\end{equation}
It has $m^d$ points, and is hence in bijection with $\{1, \dots, m^d \}$. For any $i \in \{1, \dots, m^d \}$, we identify $p_i$ with a unique point on this grid. 
By construction, we have that 
\begin{equation}
    \forall i, j \in \{1, \dots, m^d \}, \quad i \neq j \implies \| p_i - p_j\| \geq \frac{1}{m+1} \;.
\end{equation}

Now, let us consider the function $\Psi$ given by \Cref{existencecompactsupport} in dimension $d$. We note $\psi(\cdot) = a \Psi\p{\frac{\cdot}{2}}$ where $a>0$ is fixed to a small enough value such that 
\begin{equation}
\label{lkjbqsdkjhfbqsd}
    \sum_{\alpha \in \N^d : |\alpha| = \beta} \int_{[0, 1]^d} |\partial^\alpha \psi|^2 \leq L^2\;.
\end{equation}
We also define $\gamma = \int \psi$ and $\delta = \int \psi^2$.

Let $1 \geq h>0$. For any $\theta \in \{0, 1 \}^{m^d}$, we define 
\begin{equation}
\label{eq:definitionpacking}
    f_{\theta}(\cdot) \eqdef 1 + h^{\beta} \sum_{i = 1}^{m^d} \theta_i \psi \p{\frac{\cdot - p_i}{h}} - \|\theta \|_1 \gamma h^{\beta+d} \;.
\end{equation}

Let us investigate the conditions under which $f_{\theta}$ is a density of probability w.r.t. Lebesgue's measure on $[0, 1]^d$.
\begin{itemize}
    \item For any $\theta$, $f_{\theta}$ is continuous and hence measurable.
    \item $f_{\theta}$ has to be positive for any $\theta$. This is for instance the case when for any $\theta$, $\|\theta \|_1 \gamma h^{\beta+d} \leq 1$. Since $\|\theta \|_1 \leq m^d$ for any $\theta$, fixing $h = \min \left\{ \frac{1}{\gamma (m+1)}, \frac{1}{4 (m+1)}\right\}$ is enough to ensure that condition. The reason why we added the term $\frac{1}{4 (m+1)}$ in the minimum and why we took $m+1$ instead of $m$ is that we also have that for any $i$, $\psi \p{\frac{\cdot - p_i}{h}}$ has its support in $(0, 1)^d$ and that $i \neq j \implies $ $\psi \p{\frac{\cdot - p_i}{h}}$ and $\psi \p{\frac{\cdot - p_j}{h}}$ have disjoint supports.
    \item For any $\theta$, we need $\int f_{\theta}=1$, which is immediate by construction with a simple  variable swap of inverse Jacobian $h^d$ :
    \begin{equation}
        \begin{aligned}
            \int_{[0, 1]^d} f_{\theta}
            &=
            \int_{[0, 1]^d} \p{1 + h^{\beta} \sum_{i = 1}^{m^d} \theta_i \psi \p{\frac{x - p_i}{h}} - \|\theta \|_1 \gamma h^{\beta+d}} dx \\
            &=
            1 + h^{\beta} \sum_{i = 1}^{m^d} \theta_i \int_{[0, 1]^d}  \psi \p{\frac{x - p_i}{h}}  dx - \|\theta \|_1 \gamma h^{\beta+d} \\
            &\stackrel{u_i = \frac{x - p_i}{h}}{ = }
            1 + \|\theta \|_1 \gamma h^{\beta+d} - \|\theta \|_1 \gamma h^{\beta+d} \\
            &= 1
        \end{aligned}
    \end{equation}
\end{itemize}

Furthermore, we may also check that for any $\theta$, $f_{\theta} \in \mathcal{S}_L^p(\beta)$. 
\begin{itemize}
    \item For any $\theta$, by construction, the support of $\partial^{\alpha}f_{\theta}$ is included in $(0, 1)^d$ for any multi-index $\alpha$ such that $|\alpha| \geq 1$. Hence, the periodicity argument holds trivially since $\partial^{\alpha}f_{\theta} = 0$ on the boundary of $[0, 1]^d$. Furthermore, since $f_{\theta}$ is constant on the boundary of $[0, 1]^d$, the periodicity argument also holds for $f_{\theta}$.
    \item Furthermore, let us fix $\theta$ and let $\alpha$ be a multi-index such that $|\alpha| = \beta$. We have 
    \begin{equation}
    \label{kjhbdskqjfhsbkjhcvbkqsdf}
\begin{aligned}
    \int_{[0, 1]^d} \p{f_{\theta}^{(\alpha)}}^2
    &= \int_{[0, 1]^d} \p{h^{\beta} \sum_{i=1}^{m^d} \theta_i \p{x \mapsto \psi\p{\frac{x - p_i}{h}}}^{(\alpha)}}^2 \\
    &= \int_{[0, 1]^d} \p{\sum_{i=1}^{m^d} \theta_i \psi^{(\alpha)}\p{\frac{\cdot - p_i}{h}}}^2 \\
    &\stackrel{\text{disjoint supports}}{=}  \sum_{i=1}^{m^d} \theta_i \int_{[0, 1]^d} \p{ \psi^{(\alpha)}\p{\frac{\cdot - p_i}{h}}}^2 \\
    &\stackrel{\|\theta\|_1 \leq m^d \& \text{ variable swap}}{\leq} m^d h^d \int_{[0, 1]^d} \p{\psi^{(\alpha)}}^2  \\
    &\stackrel{m^d h^d \leq 1}{\leq} \int_{[0, 1]^d} \p{\psi^{(\alpha)}}^2\;,
\end{aligned}
\end{equation}
Consequently, summing over $\alpha$ yields
\begin{equation}
    \sum_{\alpha \in \N^d : |\alpha| = \beta} \int_{[0, 1]^d} \p{f_{\theta}^{(\alpha)}}^2 \leq \sum_{\alpha \in \N^d : |\alpha| = \beta} \int_{[0, 1]^d} \p{\psi^{(\alpha)}}^2 
    \stackrel{\eqref{lkjbqsdkjhfbqsd}}{\leq} L^2 \;.
\end{equation}
\end{itemize}

Now we will used what is usually referred to as Assouad's lemma, and that has been successfully used to prove lower-bounds under differential privacy in \cite{duchi2013local,duchi2013localpreprint,duchi2018minimax,acharya2021differentially}. The following result is a minor reformulation to match the notations of the article of the version that can be found in \cite{acharya2021differentially}.

\begin{fact}[Assouad's Lemma]
\label{factAssouad}
    If $(f_{\theta})$ is a family of densities of probability that is parametrized by $\theta \in \{0, 1\}^N$, and if there exists a $\tau > 0$ such that 
    \begin{equation}
    \label{kjhbkjqhsbdkf}
        \forall (\theta_1, \theta_2)\,: \quad 
        \| f_{\theta_1} - f_{\theta_2}\|^2 \geq C \tau \ham{\theta_1}{\theta_2} \;,
    \end{equation}
    then there exists an absolute constant $C >0$ such that for any estimator $\hat{f}$, by noting $\hat{\theta}$ the parameter of the closest $f_{\theta}$ in the family $(f_{\theta})_{\theta \in \{0, 1 \}^N}$ for the norm $\| \cdot \|$, then 
    \begin{equation}
    \label{jhbkqjshdf}
        \sup_{\theta \in \{0, 1 \}^N} \E_{{f_\theta}^{\otimes n}} \p{\|f_\theta - \hat{f}\|^2}
        \geq C \tau \sum_{i=1}^{N} \p{\Prob_{{{\theta_{-i}}}} (\hat{\theta}^i \neq 0) + \Prob_{{{\theta_{+i}}}} (\hat{\theta}^i \neq 1)}
    \end{equation}
    where $\Prob_{{{\theta_{+i}}}}$ and $\Prob_{{{\theta_{-i}}}}$ are the mixture distributions
    \begin{equation}
        \Prob_{{{\theta_{+i}}}} \eqdef \frac{1}{2^{N-1}} \sum_{\theta : \theta^i = 1} f_{\theta}^{\otimes n} \quad \Prob_{{{\theta_{-i}}}} \eqdef \frac{1}{2^{N-1}} \sum_{\theta : \theta^i = 0} f_{\theta}^{\otimes n}\;.
    \end{equation}
    Notice that in \eqref{jhbkqjshdf} there is a second layer or randomness that is implicit, and that is w.r.t. the estimator itself (for privacy for instance).
\end{fact}
\begin{proof}
    The proof can be found in \cite{acharya2021differentially}.
\end{proof}

We will apply this result with $N = m^d$.
First, we will check that \eqref{kjhbkjqhsbdkf} holds. 

Let $\theta_1, \theta_2$ be two parametrizations. We have that 
\begin{equation}
\label{eq:distkernelsaw}
\begin{aligned}
    \int_{[0, 1]^d} &\p{f_{\theta_1} - f_{\theta_2}}^2 \\
    &\geq 
    \sum_{i=1}^{m^d} \Ind_{\theta_1^i \neq \theta_2^i} \int_{B(p_i, h/2)}
    \p{h^{\beta + d} \p{\| \theta_2 \|_1 - \| \theta_1 \|_1} \gamma +(\theta_1^i-\theta_2^i) h^{\beta } \psi\p{\frac{t-p_i }{h}}}^2 dt\\
    &\geq
    \sum_{i=1}^{m^d} \Ind_{\theta_1^i \neq \theta_2^i} \int_{B(p_i, h/2)}
    \left\{\p{h^{\beta } \psi\p{\frac{t-p_i}{h}}}^2 \right.\\
    &\quad\quad\quad\quad\quad\quad\quad\quad\quad\quad \left.- 2 \gamma h^{2\beta +d} \left|\| \theta_1 \|_1 - \| \theta_2 \|_1\right| \psi\p{\frac{t-p_i}{h}}  \right\} dt\\ 
    &\stackrel{\text{variable swap}}{\geq}
    \ham{\theta_1}{\theta_2} h^{2 \beta + d}\p{  \delta - 2m^d h^d \gamma^2} \\
    &\stackrel{h = \min \left\{h, \frac{1}{m+1} \p{\delta / (4 \gamma^2)^{1/d}} \right\}}{\geq} \ham{\theta_1}{\theta_2} h^{2 \beta + d} \delta / 2\;,
\end{aligned}
\end{equation}
where we took the liberty to take a smaller $h$ if needed, with still a scaling proportional to $\frac{1}{m+1}$.

Then, we need to control the term $\Prob_{{{\theta_{-i}}}} (\hat{\theta}^i \neq 0) + \Prob_{{{\theta_{+i}}}} (\hat{\theta}^i \neq 1)$.

\paragraph{Privacy cost.}

First, we do so by exploiting the constraint of $\rho$-zCDP. Let us give the following lemma, which is borrowed from \cite{lalanne2023about}.
\begin{lemma}
\label{lemmalecamassouad}
    If $\hat{f}$ satisfies $\rho$-zCDP, then for any $i$,
    \begin{equation*}
    \begin{aligned}
    &\Prob_{{{\theta_{-i}}}} (\hat{\theta}^i \neq 0) + \Prob_{{{\theta_{+i}}}} (\hat{\theta}^i \neq 1)
    \geq \\
    &\quad\quad
    \frac{1}{2} \p{1 - n\sqrt{\rho/2} \frac{1}{2^{N-1}}\sum_{\theta^1, \dots, \theta^{i-1}, \theta^{i+1} \dots, \theta^{N} \in \{0, 1 \}} \tv{f_{(\theta^1, \dots, \theta^{i-1}, 0, \theta^{i+1} \dots, \theta^{N})}}{f_{(\theta^1, \dots, \theta^{i-1}, 1, \theta^{i+1} \dots, \theta^{N})}}} \;,
    \end{aligned}
\end{equation*}
where $\tv{\cdot}{\cdot}$ denotes the \emph{total variation} distance between probability measures defined as 
    \begin{equation*}
        \tv{\Prob_1}{\Prob_2} \eqdef \sup_{S \text{ measurable}} \Prob_1(S) - \Prob_2(S) \;. 
    \end{equation*}
\end{lemma} 
\begin{proof}
    Let us consider the coupling $\mathcal{C}$ that selects $\theta^1, \dots, \theta^{i-1}, \theta^{i+1} \dots, \theta^{N} \in \{0, 1 \}$ uniformly at random, and then returns a random variable that follows a conditional distribution $\mathbb{Q}_{\theta^1, \dots, \theta^{i-1}, \theta^{i+1} \dots, \theta^{N}}^{\otimes n}$ where $\mathbb{Q}_{\theta^1, \dots, \theta^{i-1}, \theta^{i+1} \dots, \theta^{N}}$ is a maximal coupling between $f_{(\theta^1, \dots, \theta^{i-1}, 0, \theta^{i+1} \dots, \theta^{N})}$ and $f_{(\theta^1, \dots, \theta^{i-1}, 1, \theta^{i+1} \dots, \theta^{N})}$, in the sense that if $X, Y \sim \mathbb{Q}_{\theta^1, \dots, \theta^{i-1}, \theta^{i+1} \dots, \theta^{N}}$, then $\Prob(X = Y) = 1 -\tv{f_{(\theta^1, \dots, \theta^{i-1}, 0, \theta^{i+1} \dots, \theta^{N})}}{f_{(\theta^1, \dots, \theta^{i-1}, 1, \theta^{i+1} \dots, \theta^{N})}}$. The existence of such coupling is well known (see, \textit{e.g.} \cite{lindvall2002lectures}).

    Then, the similarity function given by Lemma 8 in \cite{lalanne2022statistical} leads to:
    \begin{equation*}
    \begin{aligned}
    &\Prob_{{{\theta_{-i}}}} (\hat{\theta}^i \neq 0) + \Prob_{{{\theta_{+i}}}} (\hat{\theta}^i \neq 1)
    \geq 
    \frac{1}{2} \p{1 - \sqrt{\rho/2} \E_{\vect{X}, \vect{Y} \sim \mathcal{C} }\p{\ham{\vect{X}}{\vect{Y}}}} \;,
    \end{aligned}
\end{equation*}
which reduces to the advertised result.
\end{proof}

Let us fix $\theta^1, \dots, \theta^{i-1}, \theta^{i+1} \dots, \theta^{m^d} \in \{0, 1 \}$, we have that, by the classical rewriting of the total variation distance $\tv{f}{g} = \frac{1}{2} \int |f-g|$,
\begin{equation}
\label{jkhgkjhqbsdkjhfb}
    \begin{aligned}
        &\tv{f_{(\theta^1, \dots, \theta^{i-1}, 0, \theta^{i+1} \dots, \theta^{m^d})}}{f_{(\theta^1, \dots, \theta^{i-1}, 1, \theta^{i+1} \dots, \theta^{m^d})}} \\
        &\quad\quad = \frac{1}{2} \int_{[0, 1]^d} \left|f_{(\theta^1, \dots, \theta^{i-1}, 0, \theta^{i+1} \dots, \theta^{m^d})} - f_{(\theta^1, \dots, \theta^{i-1}, 1, \theta^{i+1} \dots, \theta^{m^d})}\right| \\
        &\quad\quad\leq \frac{1}{2} \int_{[0, 1]^d} \p{\gamma h^{\beta + d} + h^{\beta} \psi \p{\frac{\cdot - p_i}{h}} } \\
        &\quad\quad\stackrel{\text{variable swap}}{=} \gamma h^{\beta+d}
    \end{aligned}
\end{equation}

All in all, by combining \eqref{jkhgkjhqbsdkjhfb}, \Cref{lemmalecamassouad}, \eqref{eq:distkernelsaw} and \Cref{factAssouad}, there exist two absolute constants $C_1 > 0$ and $C_2 > 0$ such that, if $\hat{f}$ satisfies $\rho$-zCDP, then:
\begin{equation}
    \sup_{\theta \in \{0, 1 \}^N} \E_{{f_\theta}^{\otimes n}} \p{\|f_\theta - \hat{f}\|^2}
    \geq C_1 h^{2 \beta + d} m^d \delta \p{1 - C_2 \gamma n\sqrt{\rho} h^{\beta + d}} \;.
\end{equation}
Finally, choosing $h$ of the order of $\p{\gamma n \sqrt{\rho}}^{-\frac{1}{\beta + d}}$, and $m+1$ of the order of $\frac{\min \left\{1/ \gamma, 1/4 , \p{\delta / (4 \gamma^2)^{1/d}}\right\}}{h}$ complies with all the requirements on $h$ for the calculus to be valid, and allows writing that there are two quantities $C_1 > 0$ and $C_2 > 0$ depending on $L$, $\beta$ and $d$ such that, if $n \sqrt{\rho} > C_2$, then
\begin{equation}
    \sup_{\theta \in \{0, 1 \}^N} \E_{{f_\theta}^{\otimes n}} \p{\|f_\theta - \hat{f}\|^2}
    \geq C_1  \p{n \sqrt{\rho}}^{-\frac{2 \beta}{\beta + d}}\;.
\end{equation}

\paragraph{Usual sampling cost.} Without trying to exploit the private nature of the estimation, we may adopt more usual lower-bounding inequalities. 

Let us fix $\hat{f}$ and $i$. Neyman-Pearson-Le Cam's inequality (Of which the proof can be found in \cite{rigollet2015high}) allows writing
\begin{equation}
    \Prob_{{{\theta_{-i}}}} (\hat{\theta}^i \neq 0) + \Prob_{{{\theta_{+i}}}} (\hat{\theta}^i \neq 1)
    \geq 
    1 - \tv{\Prob_{{{\theta_{+i}}}}}{\Prob_{{{\theta_{-i}}}}} \;.
\end{equation}
Then, Pinsker's inequality (see for instance \cite{tsybakov2003introduction}) gives
\begin{equation}
    \Prob_{{{\theta_{-i}}}} (\hat{\theta}^i \neq 0) + \Prob_{{{\theta_{+i}}}} (\hat{\theta}^i \neq 1)
    \geq 
    1 - \sqrt{ \kl{\Prob_{{{\theta_{+i}}}}}{\Prob_{{{\theta_{-i}}}}}} \;,
\end{equation} 
where $\kl{\cdot}{\cdot}$ is the Kullback-Leibler (KL) divergence which is defined for any two probability distributions $\mathbb{P}$ and $\mathbb{Q}$ such that $\mathbb{P} \ll \mathbb{Q}$ (absolute continuity) as
\begin{equation*}
    \kl{\mathbb{P}}{\mathbb{Q}} = \int \log \p{\frac{d \mathbb{P}}{ d \mathbb{Q}}} d\mathbb{P} \;.
\end{equation*}

Then, Theorem 11 in \cite{van2014renyi} gives that 
\begin{equation*}
    \kl{\frac{1}{2^{N-1}} \sum_{\theta : \theta^i = 1} f_{\theta}^{\otimes n}}{\frac{1}{2^{N-1}} \sum_{\theta : \theta^i = 0} f_{\theta}^{\otimes n}}
    \leq \frac{1}{2^{N-1}} \sum_{\theta : \theta^i = 0} \kl{ f_{\theta^{(i \leftarrow 1)}}^{\otimes n}}{ f_{\theta^{(i \leftarrow 0)}}^{\otimes n}} \;,
\end{equation*}
where $\theta^{(i \leftarrow j)}$ means that we assign $j$ as the value of the $i^{\text{th}}$ component in $\theta$. 

Finally, by the tensorization property of the KL divergence \cite{van2014renyi} 
\begin{equation}
\label{kjhbkdjqhbsdkjfhq}
    \Prob_{{{\theta_{-i}}}} (\hat{\theta}^i \neq 0) + \Prob_{{{\theta_{+i}}}} (\hat{\theta}^i \neq 1)
    \geq 
    1 - \sqrt{\frac{1}{2^{N-1}} \sum_{\theta : \theta^i = 0} n \kl{ f_{\theta^{(i \leftarrow 1)}}}{ f_{\theta^{(i \leftarrow 0)}}}} \;.
\end{equation} 

Let us fix a $\theta$. We will upper-bound $\kl{ f_{\theta^{(i \leftarrow 1)}}}{ f_{\theta^{(i \leftarrow 0)}}}$ uniformly in $\theta$.
By definition,
\begin{equation}
    \kl{ f_{\theta^{(i \leftarrow 1)}}}{ f_{\theta^{(i \leftarrow 0)}}} 
    = 
    \int_{[0, 1]^d} \log \p{\frac{ f_{\theta^{(i \leftarrow 1)}}}{ f_{\theta^{(i \leftarrow 0)}}}}  f_{\theta^{(i \leftarrow 1)}} \;,
\end{equation}
and a classical upper bound of the KL divergence by the $\chi^2$-divergence which follows from $\log(\cdot) \leq \cdot - 1$ gives
\begin{equation}
    \kl{ f_{\theta^{(i \leftarrow 1)}}}{ f_{\theta^{(i \leftarrow 0)}}} 
    = 
    \int_{[0, 1]^d} \frac{(f_{\theta^{(i \leftarrow 1)}} - f_{\theta^{(i \leftarrow 0)}})^2}{f_{\theta^{(i \leftarrow 0)}}} \;.
\end{equation}
Notice that we took the liberty to divide by various densities of probability without justifying why they were different from $0$. We will solve this issue right now, and also control the denominator $f_{\theta^{(i \leftarrow 0)}}$ at the same time.

When we made sure that for any $\theta$, $f_{\theta}$ was always positive, we imposed that $m^d \gamma h^{\beta+d} \leq 1$. We can be more aggressive and impose that $m^d \gamma h^{\beta+d} \leq 1/2$, for instance by taking $ h \leq \frac{1}{2 \gamma (m+1)}$. This way, we have that for any $\theta$, $f_{\theta} \ge 1/2$.

As a consequence, 
\begin{equation}
\begin{aligned}
    \kl{ f_{\theta^{(i \leftarrow 1)}}}{ f_{\theta^{(i \leftarrow 0)}}} 
    &\leq
    2 \int_{[0, 1]^d} (f_{\theta^{(i \leftarrow 1)}} - f_{\theta^{(i \leftarrow 0)}})^2 \\
    &\leq 2
    \int_{[0, 1]^d} \p{\gamma h^{\beta + d} + h^{\beta} \psi \p{\frac{x - p_i}{h}}}^2 \\
    &= 2 \p{\gamma^2 h^{2 \beta + 2d} + 2\gamma^2 h^{2 \beta + 2d} + \delta h^{2 \beta + d}}\;.
\end{aligned}
\end{equation}
So, there exist $C_1>0$ and $C_2 >0$ that depend on $L$, $\beta$ and $d$ such that when $h < C_2$, then 
\begin{equation}
\begin{aligned}
    \kl{ f_{\theta^{(i \leftarrow 1)}}}{ f_{\theta^{(i \leftarrow 0)}}} 
    &\leq
    C_1 h^{2 \beta + d}\;.
\end{aligned}
\end{equation}
Furthermore, we can note that $C_1$ and $C_2$ are uniform in $\theta$.

Combining this last result with \eqref{kjhbkdjqhbsdkjfhq}, \Cref{factAssouad} and \eqref{eq:distkernelsaw}, we obtain that there exists an absolute $C_3 > 0$ such that, for any estimator $\hat{f}$,
\begin{equation}
    \sup_{\theta \in \{0, 1 \}^N} \E_{{f_\theta}^{\otimes n}} \p{\|f_\theta - \hat{f}\|^2}
    \geq C_3 h^{2 \beta + d} m^d \delta \p{1 - \sqrt{C_1 n h^{2 \beta + d}}} \;,
\end{equation}
as soon as $h < C_2$.

In the end, choosing $h$ of the order of $\p{n}^{-\frac{1}{2 \beta + d}}$, and $m+1$ of the order of $\frac{\min \left\{1/ (2\gamma), 1/4 , \p{\delta / (4 \gamma^2)^{1/d}}\right\}}{h}$ complies with all the requirements on $h$ for the calculus to be valid, and allows writing that there are two quantities $C_1 > 0$ and $C_2 > 0$ depending on $L$, $\beta$ and $d$ such that, if $n  > C_2$, then
\begin{equation}
    \sup_{\theta \in \{0, 1 \}^N} \E_{{f_\theta}^{\otimes n}} \p{\|f_\theta - \hat{f}\|^2}
    \geq C_1  n^{-\frac{2 \beta}{2 \beta + d}}\;.
\end{equation}

The two lower-bounds being valid for $\rho$-zCDP estimators, their maximum is also a lower-bound, yielding the result.

\section{Proofs of \Cref{sectionAdaptivity}}

\subsection{Proof of \Cref{theo:Lepskii_rate}}
\label{proofoftheo:Lepskii_rate}

We also define $m^*$ the integer that is associated to the closest point (from below) of the grid $\mathbb{B}_n$ to the unknown smoothness parameter $\beta$:
\begin{equation}
    \label{def:m*}
    m^* =  \min \{ m \leq k_n \, : \beta_m \leq \beta \} \quad \text{and} \quad \beta^* = \beta_{m^*}.   
\end{equation}
We emphasize that $m^*$ is a theoretical object, which is purely deterministic and not used in our adaptative procedure. We nevertheless need $m^*$ for our mathematical analysis of the Lepskii method.
For the sake of clarity, we will use the following shortcut of notations to improve the readability of our paper:
$$
\hat{f}_{M_{n,\rho'_n}(\beta_{\hat{m}_n})} = \hat{f}_{\hat{M}} \quad \text{and} \quad 
\hat{f}_{M_{n,\rho'_n}(\beta_{m^*})} = \hat{f}_{M^*} \quad \text{and} \quad 
\hat{f}_{M_{n,\rho'_n}(\beta_{\ell})} = \hat{f}_{M(\ell)},
$$
and the associated shortcut indices as well:
$$
\hat{M} = M_{n,\rho'_n}(\beta_{\hat{m}_n}) \quad \text{and} \quad M^*=M_{n,\rho'_n}(\beta_{m^*})
\quad \text{and} \quad M(\ell) = M_{n,\rho'_n}(\beta_{\ell}).
$$

To establish our adaptive result stated in Theorem \ref{theo:Lepskii_rate}, we need the next cornerstone result.

\begin{proposition}\label{prop:decomposition_adaptive}
Assume that $f \in S_L^p(\beta)$ with $n \ge e^\beta$, then 
$\hat{f}_{M_{n,\rho'_n}(\beta_{\hat{m}_n})} = \hat{f}_{\hat{M}}  $ satisfies:
\begin{align}
    \label{eq:adaptive_inequality}
 \mathbb{E}[\|\hat{f}_{\hat{M}} -f\|_2]& \leq
    2 \sqrt{   r_{n,\rho'_n}(\beta)^*  } 
  \exp\left( \frac{\varepsilon}{\beta+ d}\right) \\
  & + \sqrt{
    \sum_{\ell=0}^{k_n} r_{n,\rho'_n}(\beta_\ell)}  \sqrt{
    \sum_{\ell > m^*} \mathbb{P}\left[ \|\hat{f}_{M(\ell)}-f\|_2^2 > \frac{1}{4} r_{n,\rho'_n}(\beta_\ell)^* \right]}
\end{align}
\end{proposition}
\begin{proof}
    We  observe that the elementary decomposition holds:
    \begin{equation}
      \label{eq:dec_risk}
    \mathbb{E}[\|\hat{f}_{\hat{M}}-f\|_2] = \mathbb{E}[\|\hat{f}_{\hat{M}}-f\|_2 \mathbf{1}_{\hat{m}_n \le m^*}] + \mathbb{E}[\|\hat{f}_{\hat{M}}-f\|_2 \mathbf{1}_{\hat{m}_n > m^*}].
    \end{equation}
    We then consider the two terms separately.

    \noindent \underline{On the event $\hat{m}_n \le m^*$:}
    We apply the triangle inequality and obtain:
    $$
    \mathbb{E}[\|\hat{f}_{\hat{M}}-f\|_2 \mathbf{1}_{\hat{m}_n \le m^*}] \leq 
    \mathbb{E}\left[ \left( \|\hat{f}_{\hat{M}}-\hat{f}_{M^*}\|_2 +\|\hat{f}_{M^*}-f\|_2  \right)\mathbf{1}_{\hat{m}_n \le m^*} \right]
    $$
    Using the definition of $\hat{m}_n$ and $\hat{f}_{\hat{M}}$, we observe that \textit{almost surely}:
    \begin{align*}
    \|\hat{f}_{\hat{M}}-\hat{f}_{M^*}\|_2 \mathbf{1}_{\hat{m}_n \le m^*} & \leq \sqrt{r_{n,\rho'_n}(\beta^*)^*}\\
    & \leq \sqrt{ C (\log n)^a r_{n,\rho'_n}(\beta^*)}   \\
    & \leq \sqrt{ C  (\log n)^a r_{n,\rho'_n}(\beta)  } 
     \Bigg( \exp \left(\left( \frac{\beta}{2\beta+d}-\frac{\beta^*}{2\beta^*+d} \right)\log n \right) \\& \qquad\qquad\qquad\qquad\qquad\qquad\qquad\qquad\qquad\qquad \vee 
    \exp \left(\left( \frac{\beta}{\beta+d}-\frac{\beta^*}{\beta^*+d} \right) \frac{\log (n\sqrt{\rho'_n})}{2} \right) \Bigg)
    \\
    & \leq \sqrt{ C  (\log n)^a r_{n,\rho'_n}(\beta)  } 
    \Bigg( \exp\left( \frac{(\beta-\beta^*)d}{(2\beta+d)(2\beta^*+d)} \log n \right) 
     \\& \qquad\qquad\qquad\qquad\qquad\qquad\qquad\qquad\qquad\qquad \vee 
    \exp \left(  \frac{(\beta-\beta^*)d}{(\beta+d)(\beta^*+d)}   \frac{\log (n\sqrt{\rho'_n})}{2} \right) \Bigg)
    \\
    &\leq \sqrt{ C  (\log n)^a r_{n,\rho'_n}(\beta)  } 
    \Bigg(\exp\left( \frac{\varepsilon}{2(2\beta+ d)}\right) \vee \exp\left( \frac{\varepsilon}{\beta+ d}\right) \Bigg)\\
    & = \sqrt{ C  (\log n)^a r_{n,\rho'_n}(\beta)  } \exp\left( \frac{\varepsilon}{\beta+ d}\right),
   \end{align*}
   where we used above 
   $$
   \frac{(\beta-\beta^*)d}{(\beta+d)(\beta^*+d)}   \frac{\log (n\sqrt{\rho'_n})}{2} \leq \varepsilon \log^{-1} n \frac{d}{(\beta+d)(\beta^*+d)} \left( \log n + \frac{1}{2} \log \rho - \frac{1}{2} \log \log n \right) \leq \frac{\varepsilon}{\beta+d}.
   $$
   Obviously, the same upper bound applies when considering the expectation and we deduce that
   \begin{equation}
       \label{eq:tec_up_1}   
  \mathbb{E} \left[  \|\hat{f}_{\hat{M}}-\hat{f}_{M^*}\|_2 \mathbf{1}_{\hat{m}_n \le m^*} \right] \leq   \sqrt{ C  (\log n)^a r_{n,\rho'_n}(\beta)  } 
  \exp\left( \frac{\varepsilon}{\beta+ d}\right) = 
  \sqrt{ r_{n,\rho'_n}(\beta)^*  }   \exp\left( \frac{\varepsilon}{\beta+ d}\right)
  .
   \end{equation}
   The second term is dealt easily  using the non-adaptive rate of convergence of $\hat{f}_{M^*}$, regardless the value of $m^*$ with respect to $\hat{m}_n$,
   and the Cauchy-Schwarz inequality:
   $$
\mathbb{E}  \left[ \|\hat{f}_{M^*}-f\|_2   \mathbf{1}_{\hat{m}_n \le m^*} \right] \leq 
\sqrt{\mathbb{E}  \left[  \|\hat{f}_{M^*}-f\|_2^2 \right]} \leq \sqrt{r_{n,\rho'_n}(\beta^*)}
   $$
   Using the same arguments as above, we obtain similarly:
     \begin{equation}
       \label{eq:tec_up_2}   
 \mathbb{E}  \left[  \|\hat{f}_{M^*}-f\|_2  \mathbf{1}_{\hat{m}_n \le m^*} \right]
    \leq \sqrt{   r_{n,\rho'_n}(\beta)  } 
  \exp\left( \frac{\varepsilon}{\beta+ d}\right).
  \end{equation}
  We now gather Equations \eqref{eq:tec_up_1} and \eqref{eq:tec_up_2} and obtain that:
  \begin{equation}
      \label{eq:dec_risk_sol_1}
\mathbb{E}\left[\|\hat{f}_{\hat{M}}-f\|_2 \mathbf{1}_{\hat{m}_n \le m^*}\right] \leq 
2 \sqrt{   r_{n,\rho'_n}(\beta)^*  } 
  \exp\left( \frac{\varepsilon}{\beta+ d}\right).
  \end{equation}

    \noindent \underline{On the event $\hat{m}_n > m^*$:}  
We still apply the triangle inequality and observe that for any pair $(M,M')$:
$$
\|\hat{f}_{M}-\hat{f}_{M'}\|_2 \leq \|\hat{f}_M-f\|_2 + \|\hat{f}_{M'}-f\|_2.
$$
Consequently, we have
\begin{align*}
    \left\{ \hat{m}_n > m^*\right\}  
    &= 
    \left\{\exists \ell > m^* \, : \,  \|\hat{f}_{M(\ell)}-\hat{f}_{M^*}\|_2 > \sqrt{r_{n,\rho'_n}(\beta_\ell)^*}\right\} \\
    & \subset 
    \left\{\exists \ell > m^* \, : \,  \|\hat{f}_{M(\ell)}-f\|_2 + \|\hat{f}_{M^*}-f\|_2 > \sqrt{r_{n,\rho'_n}(\beta_\ell)^*}\right\} \\
    & \subset 
    \left\{\exists \ell > m^* \, : \,  \|\hat{f}_{M(\ell)}-f\|_2 > \frac{1}{2} \sqrt{r_{n,\rho'_n}(\beta_\ell)^*} \right\}
    \cup \left\{ \exists \ell > m^* \, : \, 
    \|\hat{f}_{M^*}-f\|_2 > \frac{1}{2} \sqrt{r_{n,\rho'_n}(\beta_\ell)^*}\right\} \\
    & \subset 
    \left\{\exists \ell > m^* \, : \,  \|\hat{f}_{M(\ell)}-f\|_2 > \frac{1}{2} \sqrt{r_{n,\rho'_n}(\beta_\ell)^*} \right\}
    \cup \left\{ 
    \|\hat{f}_{M^*}-f\|_2 > \frac{1}{2} \sqrt{r_{n,\rho'_n}(\beta_{m^*})^*}\right\},
\end{align*}
where the last inequality comes from the monotonicity (decreasing function) of $\beta \longmapsto r_{n,\rho'_n}(\beta)$. We then deduce with a union bound that:
\begin{align}
    \label{eq:majo_proba}
    \mathbb{E}\left[ \mathbf{1}_{\left\{ \hat{m}_n > m^*\right\}} \right] 
    \leq \sum_{\ell > m^*} \mathbb{P}\left[ \|\hat{f}_{M(\ell)}-f\|_2^2 >  \frac{1}{4} r_{n,\rho'_n}(\beta_\ell)^* \right]
\end{align}
We then use the Cauchy-Schwarz inequality and \eqref{eq:majo_proba} to obtain:
\begin{align*}
    \mathbb{E}[\|\hat{f}_{\hat{M}}-f\|_2 \mathbf{1}_{\hat{m}_n > m^*}]& \leq 
    \sqrt{\mathbb{E}[\|\hat{f}_{\hat{M}}-f\|_2^2 ]} \sqrt{\mathbb{E}[\mathbf{1}_{\hat{m}_n > m^*}]} \\
    & \leq  \sqrt{\mathbb{E}\left[\sum_{\ell=0}^{k_n} \|\hat{f}_{M(\ell)}-f\|_2^2 \right]} \sqrt{
    \sum_{\ell > m^*} \mathbb{P}\left[ \|\hat{f}_{M(\ell)}-f\|_2^2 >   \frac{1}{4} r_{n,\rho'_n}(\beta_\ell)^* \right]} \\
    & \leq  \sqrt{
    \sum_{\ell=0}^{k_n} r_{n,\rho'_n}(\beta_\ell)}  \sqrt{
    \sum_{\ell > m^*} \mathbb{P}\left[ \|\hat{f}_{M(\ell)}-f\|_2^2 > \frac{1}{4} r_{n,\rho'_n}(\beta_\ell)^* \right]}
\end{align*}

\end{proof}

From Proposition  \ref{prop:decomposition_adaptive}, we observe that the upper bound of the risk of our adaptive procedure depends on two terms. The first one involves the risk $r_{n,\rho}(\beta)$, up to some multiplicative $\log n$ term, while the second term will be shown to be negligible with respect to the first one as soon as $a$ and $C$ are suitably chosen (see Definition \eqref{def:penalty_risk}).

The next proposition is purely technical and does not involve any statistical insight.
\begin{proposition}
    \label{prop:risk_series}
Assume that $\varepsilon \leq 1/2$, then for any $\rho>0, n \ge 1$ and $d \ge 1$:
$$\sum_{\ell=0}^{k_n} r_{n,\rho'_n}(\beta_\ell) \leq 
4 (2+d) \varepsilon^{-1}  \log n^2   \left({\rho'_n}^{-\frac{ 1 }{1  + d}}+2\right).
$$    
\end{proposition}
\begin{proof}
We observe from our definition of $r_{n,\rho'_n}(\beta)$ that:
    \begin{align*}
        \sum_{\ell=0}^{k_n} r_{n,\rho'_n}(\beta_\ell)& = \sum_{\ell=0}^{k_n} 
 \left(
        n^{-\frac{2 \beta_\ell}{2 \beta_\ell  + d}} + (n \sqrt{\rho'_n})^{- \frac{2 \beta_\ell}{\beta_\ell+d}})
        \right) \\
        & = \sum_{\ell=0}^{  \lfloor \varepsilon^{-1} \log^2 n \rfloor  } n^{-\frac{ \ell \varepsilon/\log n }{ \ell \varepsilon/\log n   + d/2}} + (n \sqrt{\rho'_n})
        ^{-\frac{ \ell \varepsilon/\log n }{ \ell \varepsilon/2\log n   + d/2}} \\
        & = \sum_{\ell \varepsilon < \log n} n^{-\frac{ \ell \varepsilon/\log n }{ \ell \varepsilon/\log n   + d/2}} + (n \sqrt{\rho'_n})
        ^{-\frac{ \ell \varepsilon/\log n }{ \ell \varepsilon/2\log n   + d/2}} +  \sum_{\ell \ge  \lfloor  \varepsilon^{-1}\log n \rfloor }^{\lfloor \varepsilon^{-1} \log^2 n \rfloor} n^{-\frac{ \ell \varepsilon/\log n }{ \ell \varepsilon/\log n   + d/2}} + (n \sqrt{\rho'_n})
        ^{-\frac{ \ell \varepsilon/\log n }{ \ell \varepsilon/2\log n   + d/2}}
            \end{align*}
We focus on the first sum and observe that when $\ell \varepsilon < \log n$:
$$
n^{-\frac{ \ell \varepsilon/\log n }{ \ell \varepsilon/\log n   + d/2}} = e^{-\frac{ \ell \varepsilon/\log n }{ \ell \varepsilon/\log n   + d/2} \log n} =  e^{-\frac{ \ell \varepsilon }{ \ell \varepsilon/\log n   + d/2} } \leq
 e^{-\frac{ \ell \varepsilon }{ 1   + d/2} },
$$
and similarly:
$$
 (n \sqrt{\rho'_n})
        ^{-\frac{ \ell \varepsilon/\log n }{ \ell \varepsilon/2\log n   + d/2}} = 
        e^{-\frac{ \ell \varepsilon/\log n }{ \ell \varepsilon/2\log n   + d/2} \log n} 
        {\rho'_n}^{-\frac{ \ell \varepsilon/\log n }{ \ell \varepsilon/\log n   + d}} \leq 
        e^{-\frac{ 2 \ell \varepsilon }{ 1   + d }} {\rho'_n}^{-\frac{ 1 }{1  + d}} .
        $$
Hence, using a geometric series, we get:
\begin{align}    
    \sum_{\ell \varepsilon < \log n} n^{-\frac{ \ell \varepsilon/\log n }{ \ell \varepsilon/\log n   + d/2}} + (n \sqrt{\rho'_n})
        ^{-\frac{ \ell \varepsilon/\log n }{ \ell \varepsilon/2\log n   + d/2}} &\leq \sum_{\ell=0}^{+\infty}
          e^{-\frac{ \ell \varepsilon }{ 1   + d/2} }+e^{-\frac{ 2 \ell \varepsilon }{ 1   + d }} {\rho'_n}^{-\frac{ 1 }{1  + d}}  \nonumber \\
          &= \frac{1}{1-e^{-\frac{ \varepsilon }{ 1   + d/2} }} + \frac{{\rho'_n}^{-\frac{ 1 }{1  + d}}}{1-e^{-\frac{ 2 \varepsilon }{ 1   + d} }} \nonumber\\
          & \leq 4 (2+d) \varepsilon^{-1} {\rho'_n}^{-\frac{ 1 }{1  + d}}.
          \label{eq:term_1}
\end{align}
where the last line comes from the bound $e^{-t} \leq 1-t/2$ when $t \in [0,1/2)$.

Concerning now the second sum, when $\ell \ge  \varepsilon^{-1}\log n$, we verify that:
$$
\ell \ge  \varepsilon^{-1}\log n \Longrightarrow \frac{ \ell \varepsilon/\log n }{ \ell \varepsilon/\log n   + d/2}>\frac{2}{2+d} \quad \text{and} \quad \frac{ \ell \varepsilon/\log n }{ \ell \varepsilon/2\log n   + d/2} > \frac{2}{1+d},
$$
which in turn implies that
\begin{equation}
    \label{eq:term_2}
 \sum_{\ell \ge  \varepsilon^{-1}\log n}^{k_n} n^{-\frac{ \ell \varepsilon/\log n }{ \ell \varepsilon/\log n   + d/2}} + (n \sqrt{\rho'_n})
        ^{-\frac{ \ell \varepsilon/\log n }{ \ell \varepsilon/2\log n   + d/2}} <  
        \varepsilon^{-1} \log^2 n \left(n^{-\frac{2}{2+d} }  + (n \sqrt{\rho'_n})^{-\frac{2}{1+d} }\right)
\end{equation}
 Gathering Equations \eqref{eq:term_1} and \eqref{eq:term_2} yields the bound independent from $n$ and $d$ as soon as $\varepsilon< 1/2$:
$$
\sum_{\ell=0}^{k_n} r_{n,\rho'_n}(\beta_\ell) \leq 
4 (2+d) \varepsilon^{-1}  \log n^2   \left({\rho'_n}^{-\frac{ 1 }{1  + d}}+2\right)
$$
\end{proof}

We finally upper bound the second term of \eqref{eq:adaptive_inequality} that involves $\mathbb{P}\left[ \|\hat{f}_{M(\ell)}-f\|_2^2 >  \frac{1}{4} r_{n,\rho}(\beta_\ell)^*\right]$, to be studied when $\ell>m^*$. We obtain the next result.

\begin{proposition}
    \label{prop:series_probability}
    Assume that $C>8L^2 \vee 2^{2d+10}$, that $a\ge 1$ and $n \ge 3$, then 
    $$
    \sqrt{\sum_{l > m^*} \mathbb{P}\left[ \|\hat{f}_{M(\ell)}-f\|_2^2 >  \frac{1}{4} r_{n,\rho'_n}(\beta_\ell)^*\right]} \leq \sqrt{2   \varepsilon^{-1} } \,\log n \, n^{-2}.$$
\end{proposition}

\begin{proof}
    
We first consider any integer $\ell>m^*$ and our starting point is  the Parseval equality: we decompose the loss between $\hat{f}_{M(\ell)}$ and $f$ as follows:
\begin{align*}
    \|\hat{f}_{M(\ell)}-f\|_2^2 & =\|\hat{f}_{M(\ell)}-f_{M(\ell)}\|_2^2 + \|f_{M(\ell)}-f\|_2^2 \\
    & \leq 2\left( \sum_{k \in \{-M(\ell),...,M(\ell)\}^d\}}|\theta_k-\tilde{\theta}_k|^2 
    + \sigma_{M(\ell)}^2 \sum_{k \in \{-M(\ell),...,M(\ell)\}^d\}} |\xi_k|^2\right) + \frac{L^2}{(2\pi)^{2 \beta}} (M(\ell)+1)^{-2\beta},
\end{align*}
where in the last line we used the tail upper bound of the Fourier series on Sobolev spaces stated in Lemma \ref{lemmaBiasUBSobolev}.

We observe with our alleviated notations, we obtain that:
$$
\frac{1}{4}r_{n,\rho'_n}(\beta_\ell)^* = \frac{C}{4} (\log n)^a M_{n,\rho'_n}(\beta_\ell)^{-2\beta_\ell} = \frac{C}{4} (\log n)^{a} M(\ell)^{-2\beta_\ell} > \frac{C}{4} (\log n)^{a} (M(\ell)+1)^{-2\beta_\ell}.
$$
Hence, when $\ell>m^*$, we get $\beta_\ell<\beta_{m^*}<\beta$, which implies 
$(M(\ell)+1)^{-2\beta_\ell} > (M(\ell)+1)^{-2\beta}$.
Therefore, as soon as $\frac{C}{4}>2 L^2$, we have:
$$
\frac{L^2}{(2\pi)^{2 \beta}} (M(\ell)+1)^{-2\beta} < \frac{1}{8}r_{n,\rho'_n}(\beta_\ell)^*.
$$
For a such choice of $C$, we then obtain that for any $a>0$ and any $n \ge 3$:
\begin{align*}
&\left\{\|\hat{f}_{M(\ell)}-f\|_2^2 > \frac{1}{4}r_{n,\rho'_n}(\beta_\ell)^* \right\} \\&\quad\quad \subset 
\left\{   \sum_{k \in \{-M(\ell),...,M(\ell)\}^d\}}|\theta_k-\tilde{\theta}_k|^2 
    + \sigma_{M(\ell)}^2 \sum_{k \in \{-M(\ell),...,M(\ell)\}^d\}} |\xi_k|^2  >  \frac{1}{16}r_{n,\rho'_n}(\beta_\ell)^*\right\}    \\
    &\quad\quad \subset 
\underbrace{ \left\{   \sum_{k \in \{-M(\ell),...,M(\ell)\}^d\}}|\theta_k-\tilde{\theta}_k|^2  > 
     \frac{1}{32}r_{n,\rho'_n}(\beta_\ell)^*\right\}}_{:=E_1} \\
      &\quad\quad \qquad \cup 
     \underbrace{\left\{  \sigma_{M(\ell)}^2 \sum_{k \in \{-M(\ell),...,M(\ell)\}^d\}} |\xi_k|^2  > 
     \frac{1}{32}r_{n,\rho'_n}(\beta_\ell)^*\right\} }_{:=E_2}.
\end{align*}
We now consider $E_1$ and $E_2$ separately.

\underline{Study of $E_1$: concentration of the sequence $(\tilde{\theta}_k)_{k \in \mathbb{Z}^d}$.}
We use a simple union bound:
$$
E_1 \subset \bigcup_{k \in \{-M(\ell),...,M(\ell)\}^d}
\left\{ |\theta_k-\tilde{\theta}_k|^2 \ge  \frac{r_{n,\rho'_n}(\beta_\ell)^*}{32 (2M(\ell)+1)^d} \right\}.
$$
The Hoeffding inequality applied to the (complex) bounded sequence $(\tilde{\theta}_k)_{k\in \mathbb{Z}^d}$ yields
$$
\forall t >0 \qquad \mathbb{P} (|\theta_k-\tilde{\theta}_k|^2 \ge t) \leq 4 e^{-n t^2/4}.
$$
 
Applying this previous inequality in the union bound above leads to
\begin{align*}
    \mathbb{P}(E_1) & \leq 4 (2M(\ell)+1)^d e^{- n \frac{r_{n,\rho'_n}(\beta_\ell)^*}{128 (2M(\ell)+1)^d} } \\
    & = 
    4 (2M(\ell)+1)^d e^{-n \frac{C (\log n)^a M(\ell)^{-2 \beta_{\ell}}}{128 (2M(\ell)+1)^d}} \\
    & \leq  4 (2M(\ell)+1)^d e^{-n \frac{C (\log n)^a M(\ell)^{-(2 \beta_{\ell}+d)}}{2^d 128}}.
\end{align*}
Using Equation \eqref{def:cut_off}, we observe that $n M(\ell)^{2 \beta_\ell+d} \ge 1$, which entails:
$$
\mathbb{P}(E_1) \leq 4 (2M(\ell)+1)^d e^{- \frac{C (\log n)^a}{2^d 128 }}.
$$
Then, using that $a>1$ and remarking from Equation \eqref{def:cut_off} that $M(\ell)^d \leq n$, we deduce thanks to our choice of $C$ that:
\begin{equation}
    \label{eq:upper_bound_E1}
    \mathbb{P}(E_1) \leq 2^{d+2}  n^{1-\frac{C}{2^{d+6}}} \leq 2^{d+2} n^{-2^d-4} \leq n^{-4}.
\end{equation}

\underline{Study of $E_2$: concentration of the $\chi^2$ noise of privacy.} From the definition of $(\xi_k)_{k \in \mathbb{Z}^d}$ as a complex Gaussian random variable, we now that 
$$\sum_{k \in \{-M(\ell),...,M(\ell)\}^d\}} |\xi_k|^2  \sim \chi^2(2(2M(\ell)+1)^d),$$
and centering the chi square distribution yields:
\begin{align*}
\mathbb{P}(E_2)   &= \mathbb{P}\left(\sigma_{M(\ell)}^2 \chi^2(2(2M(\ell)+1)^d) > \frac{r_{n,\rho'_n}(\beta_\ell)^*}{32}\right)\\
&= 
\mathbb{P}\left( \chi^2(2(2M(\ell)+1)^d) - 2(2M(\ell)+1)^d) > \frac{r_{n,\rho'_n}(\beta_\ell)^*}{32 \sigma_{M(\ell)}^2}
- 2(2M(\ell)+1)^d)\right).
\end{align*}
Using that the variance factor needs to be tuned as  $\sigma_{M(\ell)}=\frac{2 \sqrt{(2M(\ell)+1)^d}}{n\sqrt{\rho'_n}}$ to ensure a $\rho-zCDP$ and the value of $r_{n,\rho'_n}(\beta_\ell)^*$ stated in \eqref{def:penalty_risk}, we can expand the right hand side of the last inequality as:
\begin{align*}
    \frac{r_{n,\rho'_n}(\beta_\ell)^*}{32 \sigma_{M(\ell)}^2}
- 2(2M(\ell)+1)^d) &=2 (2M(\ell)+1)^d) \left( \frac{C (\log n)^a r_{n,\rho'_n}(\beta_\ell) n^2 \rho'_n}{256 (2M(\ell)+1)^{2d}) } -1\right) \\
&= 2 (2M(\ell)+1)^d) \left( \frac{C }{4^d  256 } (\log n)^a M(\ell)^{-2(\beta+d)} n^2 \rho'_n -1\right) \\
&\ge  2 (2M(\ell)+1)^d) \left( \frac{C }{4^d  256 } (\log n)^a  -1\right),
\end{align*}
where the last line comes from the definition of $M(\ell)$ that guarantees 
$$
M(\ell)^{2 (\beta_{\ell}+d)} \leq n^2 \rho'_n.
$$
We may choose $C \ge 4^d 512$, define $D=2(2M(\ell)+1)^d$ and we observe that the probability of $E_2$ is upper bounded by:
$$
\mathbb{P}(E_2) \leq \mathbb{P}\left(\chi^2(D)-D \ge  \frac{C}{2} D (\log n)^a\right).
$$
We now use the $\chi^2$ concentration upper bound stated in Equation \eqref{factChiSquaredConcentration} with $\sigma=1$ and $\delta = \frac{C}{2} (\log n)^a$ and obtain that:
\begin{equation}\label{eq:upper_bound_E2}
\mathbb{P}(E_2) \leq e^{-D \frac{C^2 (\log n)^{2a}}{16} } \vee e^{-D \frac{C  (\log n)^a}{4}} \leq 
 e^{- \frac{C  (\log n)^a}{2}} \leq n^{-C/2} \leq n^{-5},
\end{equation}
according to $a \ge 1$, $D \ge 2$ and our choice of $C$ in the statement of the proposition.
\end{proof}

\subsection{Proof of \Cref{theoremAdaptiveSelectionUtility}}
\label{proofOfTheoremAdaptiveSelectionUtility}

First, we can notice that the claim about the privacy of the whole estimation procedure is a direct consequence of \Cref{factCompositionConcentratedDifferentialPrivacy}. The rest of this proof only focuses on the utility claim.

Let us note $\rho' =  \rho / |\mathcal{M}|$
We start by writing $\Lambda^{(1)}(\cdot)$ as a sum of two terms: a sampling one  and a privacy one:
\begin{equation}
     \Lambda^{(1)}(M) \eqdef \Lambda^{(1)}_{\text{samp}}(M) + \Lambda^{(1)}_{\text{priv}}(M) \quad \forall M \in \mathcal{M} \;,
\end{equation}
and $\Lambda^{(2)}(\cdot)$ as the sum of $\Lambda^{(1)}(\cdot)$ and of a privacy term
\begin{equation}
     \Lambda^{(2)}(M) \eqdef \Lambda^{(1)}(M) + \Delta_{\text{priv}}(M) \quad \forall M \in \mathcal{M} \;.
\end{equation}

The values of $\Lambda^{(1)}_{\text{samp}}(\cdot)$, $\Lambda^{(1)}_{\text{priv}}(\cdot)$ and $\Delta_{\text{priv}}(\cdot)$ will be fixed later in the proof.

Then, for any $M$, 
\begin{equation}
\begin{aligned}
    \|\hat{f}_{\hat{M}} - f \|^2
    &\leq 
    3 \p{ \|\hat{f}_{\hat{M}} - \proj_{S_{M}}(\hat{f}_{\hat{M}}) \|^2 + \| \proj_{S_{M}}(\hat{f}_{\hat{M}})  - \hat{f}_{{M}}\|^2 + \|\hat{f}_{{M}} - f\|^2} \\
    &\leq 
    6 \Bigg( \|\hat{f}_{\hat{M}} - \proj_{S_{\hat{M}}}(\hat{f}_{{M}}) \|^2 + \| \proj_{S_{{M}}}(\hat{f}_{\hat{M}})  - \hat{f}_{{M}}\|^2 + \| \proj_{S_{M}}(\hat{f}_{\hat{M}})  - \proj_{S_{\hat{M}}}(\hat{f}_{{M}})\|^2 \\ 
    &\qquad\qquad\qquad\qquad\qquad\qquad\qquad\qquad\qquad\qquad\qquad\qquad\qquad\qquad\qquad\qquad+ \|\hat{f}_{{M}} - f\|^2 \Bigg)  \;.
\end{aligned}
\end{equation}
Because of the definition of $B^2(\cdot)$, we may write that
\begin{equation}
\begin{aligned}
    \|\hat{f}_{\hat{M}} - f \|^2
    &\leq 
    6 \Bigg( B^2(M) +  \Lambda^{(1)}(\hat{M}) + B^2(\hat{M}) +  \Lambda^{(1)}(M) + \| \proj_{S_{M}}(\hat{f}_{\hat{M}})  - \proj_{S_{\hat{M}}}(\hat{f}_{{M}})\|^2 + \|\hat{f}_{{M}} - f\|^2 \Bigg) \;, 
\end{aligned}
\end{equation}
which gives, because of the relation linking $\Lambda^{(1)}(\cdot)$ and $\Lambda^{(2)}(\cdot)$,
\begin{equation}
\begin{aligned}
    \|\hat{f}_{\hat{M}} - f \|^2
    &\leq 
    6 \Bigg( B^2(M) +  \Lambda^{(2)}(\hat{M}) + B^2(\hat{M}) +  \Lambda^{(2)}({M}) \\ &\quad\quad\quad\quad+ \p{\| \proj_{S_{M}}(\hat{f}_{\hat{M}})  - \proj_{S_{\hat{M}}}(\hat{f}_{{M}})\|^2 - (\Delta_{\text{priv}}(M) + \Delta_{\text{priv}}(\hat{M}))} + \|\hat{f}_{{M}} - f\|^2\Bigg) \;.
\end{aligned}
\end{equation}
Finally, because of the selection rule of $\hat{M}$,
\begin{equation}
\begin{aligned}
    \|\hat{f}_{\hat{M}} - f \|^2
    &\leq 
    6 \Bigg( 2 (B^2(M) + \Lambda^{(2)}({M})) \\ &\quad\quad\quad\quad+ 
    \underbrace{\p{\| \proj_{S_{M}}(\hat{f}_{\hat{M}})  - \proj_{S_{\hat{M}}}(\hat{f}_{{M}})\|^2 - (\Delta_{\text{priv}}(M) + \Delta_{\text{priv}}(\hat{M}))}_+}_{\text{\red{Extra term 1}}} + 
    \|\hat{f}_{{M}} - f\|^2\Bigg) \;.
\end{aligned}
\end{equation}
We recall that this holds for any $M \in \mathcal{M}$.
Furthermore, in order to have control on $B^2(\cdot)$, we may write that for any model $\mathcal{M}' \in \mathcal{M}$,
\begin{equation}
    \begin{aligned}
            \|\proj_{S_{M'}}&(\hat{f}_M) - \hat{f}_{M'} \|^2 - \Lambda^{(1)}(M') \\
        &\leq  2 \p{\|\proj_{S_{M'}}(\tilde{f}_M) - \tilde{f}_{M'} \|^2 + \| \proj_{S_{M'}}( (\hat{f}_M - \tilde{f}_M)) -  ( \hat{f}_{M'} - \tilde{f}_{M'})) \|^2} - \Lambda^{(1)}(M') \\
        &\leq  6 \bigg(\|\tilde{f}_{M'} - {f}_{M'} \|^2 + \|\proj_{S_{M'}}(\tilde{f}_M) - {f}_{M \wedge M'} \|^2  + \|{f}_{M'} - {f}_{M \wedge M'} \|^2 \\ 
        &\quad\quad\quad\quad\quad\quad\quad\quad+ \| \proj_{S_{M'}}( (\hat{f}_M - \tilde{f}_M)) -  ( \hat{f}_{M'} - \tilde{f}_{M'})) \|^2 \bigg) - \Lambda^{(1)}(M')  \;.
    \end{aligned}
\end{equation}
Then using that $\proj_{S_{M'}}(\tilde{f}_M) = \tilde{f}_{M \wedge M'}$ and that $\|{f}_{M'} - {f}_{M \wedge M'} \|^2 \leq \|f - {f}_{M}\|^2$ (which is easily seen using the Parseval formula),
\begin{equation}
    \begin{aligned}
            \|\proj_{S_{M'}}&(\hat{f}_M) - \hat{f}_{M'} \|^2 - \Lambda^{(1)}(M') \\
        &\leq  6 \bigg(\|\tilde{f}_{M'} - {f}_{M'} \|^2 + \|\tilde{f}_{M \wedge M'} - {f}_{M \wedge M'} \|^2  + \|f - {f}_{M}\|^2 \\ 
        &\quad\quad\quad\quad\quad\quad\quad\quad+ \| \proj_{S_{M'}}( (\hat{f}_M - \tilde{f}_M)) -  ( \hat{f}_{M'} - \tilde{f}_{M'})) \|^2 \bigg) - \Lambda^{(1)}(M') \\
        &\leq  6 \bigg(2 \|\tilde{f}_{M'} - {f}_{M'} \|^2 + \|f - {f}_{K}\|^2 + 2\| \proj_{S_{M'}}( (\hat{f}_M - \tilde{f}_M))\|^2 + 2\| ( \hat{f}_{M'} - \tilde{f}_{M'})) \|^2 \bigg) - \Lambda^{(1)}(M') \;.
    \end{aligned}
\end{equation}
Finally, the decomposition of $\Lambda^{(1)}(\cdot)$ yields
\begin{equation}
    \begin{aligned}
            \|\proj_{S_{M'}}&(\hat{f}_M) - \hat{f}_{M'} \|^2 - \Lambda^{(1)}(M') \\
        &=  6 \bigg(2 \p{\|\tilde{f}_{M'} - {f}_{M'} \|^2  - \frac{\Lambda^{(1)}_{\text{samp}}(M')}{12}} + \|f - {f}_{M}\|^2  + 2\| \proj_{S_{M'}}( (\hat{f}_M - \tilde{f}_M))\|^2 \\
        &\quad\quad\quad\quad\quad\quad\quad\quad+ 2\p{\| \hat{f}_{M'} - \tilde{f}_{M'}) \|^2 - \frac{\Lambda^{(1)}_{\text{priv}}(M')}{12} } \bigg) \\
        &\leq  6 \bigg(2 \underbrace{\p{\|\tilde{f}_{M'} - {f}_{M'} \|^2  - \frac{\Lambda^{(1)}_{\text{samp}}(M')}{12}}_+}_{\red{\text{Extra term 2}}} +  \|f - {f}_{M}\|^2  + 2\| \hat{f}_M - \tilde{f}_M\|^2\\
        &\quad\quad\quad\quad\quad\quad\quad\quad+ 2\underbrace{\p{\| \hat{f}_{M'} - \tilde{f}_{M'} \|^2 - \frac{\Lambda^{(1)}_{\text{priv}}(M')}{12} }_+}_{\red{\text{Extra term 3}}} \bigg) \;.
    \end{aligned}
\end{equation}

We thus have decomposed the problem in quantities that we can perfectly control, and with two extra terms that we have to control. This is where the penalization terms are useful in order to force the exponential convergence.

\paragraph{Control of the extra term 2.}

This term is handled with the help of the Talagrand inequality \cite{Talagrand1996InMat.126..505T,Ledoux_PS_1997__1__63_0,Klein10.1214/009117905000000044}, using a strategy close to the one presented in \cite{Non_Paramatric_Comte_2017}.

Let $\mathcal{M}' \in \mathfrak{M}$, we have that 
\begin{equation*}
    \|\tilde{f}_{M'} - {f}_{M'} \|^2 = \sup_{g : \|g\| \leq 1} |\langle \tilde{f}_{M'} - {f}_{M'}, g\rangle|^2 \;.
\end{equation*}
Furthermore, by separability of $L^2$ and the fact that $g \mapsto |\langle \tilde{f}_{M'} - {f}_{M'}, g\rangle|^2$ is continuous, we may consider this supremum over a \emph{countable} family of functions only (for applying \Cref{factTalagrandInequality}).

For any $g$ such that $\|g \| \leq 1$,
\begin{equation}
    \begin{aligned}
        \langle \tilde{f}_{M'} - {f}_{M'}, g\rangle
        &= \left\langle \frac{1}{n} \sum_{i=1}^n \p{ \sum_{k \in \{-M, \dots, M \}^d} \bar{\phi}_k(X_i)  {\phi}_k - f_{K'}}, g \right\rangle \\
        &=  \frac{1}{n} \sum_{i=1}^n \p{ 
        \underbrace{\sum_{k \in \{-M, \dots, M \}^d} \left\langle{\phi}_k , g \right\rangle \bar{\phi}_k(X_i)}_{\defeq T_g^{(K')}(X_i)}
        - 
        \underbrace{\left\langle f_{K'} , g \right\rangle}_{= \E \p{T_g^{(M')}(X_i)}}} \\
        &= \nu_n(T_g^{(M')}) \;,
    \end{aligned}
\end{equation}
where $\nu_n(T_g^{(M')})$ is defined in \Cref{factTalagrandInequality}.

We may thus rewrite 
\begin{equation*}
    \|\tilde{f}_{M'} - {f}_{M'} \|^2 = \sup_{g : \|g\| \leq 1} |\nu_n(T_g^{(M')})|^2 \;,
\end{equation*}
where the $\sup$ may be restricted to a countable family. However, $\nu_n(T_g^{(M')})$ is \emph{not} a real-valued quantity, and we cannot apply \Cref{factTalagrandInequality} directly. We will have to resort to decompose the quantities of interest and to add an extra factor $2$ at the end since 
\begin{equation*}
    |\nu_n(T_g^{(M')})|^2 = R(\nu_n(T_g^{(M')}))^2 + I(\nu_n(T_g^{(M')}))^2 = \nu_n(R(T_g^{(M')}))^2  + \nu_n(I(T_g^{(M')}))^2 \;,
\end{equation*}
and since taking the real part or the imaginary part are contractive projections, and hence reduce the quantities such as the modulus and the variance.

We may first see that
\begin{equation}
    \begin{aligned}
         \| T_g^{(M')}\|^2 
        &= \sum_{k \in \{-M', \dots, M' \}^d} |\left\langle{\phi}_k , g \right\rangle|^2 \\ 
        &\leq \| g\|^2 \\
        &\leq 1 
    \end{aligned}
\end{equation}
because $\|g\| \leq 1$.

Hence, we may write 
\begin{equation}
    T_g^{(M')} = \sum_{k \in \{-M', \dots, M' \}^d} \alpha_k \phi_k
\end{equation}
where
\begin{equation}
    \sum_{k \in \{-M', \dots, M' \}^d} |\alpha_k|^2 \leq 1.
\end{equation}

Then, 
\begin{equation}
    \begin{aligned}
    |\nu_n(T_g^{(M')})|^2 
        &\leq \left|\sum_{k \in \{-M', \dots, M' \}^d} \alpha_k \nu_n(\phi_k) \right|^2 \\
        &\stackrel{\text{Cauchy-Schwarz}}{\leq} \p{\sum_{k \in \{-M', \dots, M' \}^d} |\alpha_k|^2} \p{\sum_{k \in \{-M', \dots, M' \}^d} |\nu_n(\phi_k)|^2}\\
        &\stackrel{}{=} \sum_{k \in \{-M', \dots, M' \}^d} |\nu_n(\phi_k)|^2 \;,
    \end{aligned}
\end{equation}
which in turn gives that
\begin{equation}
    \begin{aligned}
    \max_{P(\cdot) = R(\cdot) \text{ or } I(\cdot)} \left\{ \p{\E \p{\sup_{T_g^{(M')} : \|g\| \leq 1} |\nu_n \p{P(T_g^{(M')})} |}}^2\right\}
        &\leq\p{\E \p{\sup_{T_g^{(M')} : \|g\| \leq 1} |\nu_n \p{T_g^{(M')}} |}}^2\\
        &\stackrel{\text{Jensen}}{\leq} \E \p{\sup_{T_g^{(K')} : \|g\| \leq 1} \p{\mu_n \p{t} }^2} \\
        &\leq \E \p{ \sum_{k \in \{-M', \dots, M' \}^d} |\nu_n(\phi_k)|^2} \\
        &= \frac{1}{n} \sum_{k \in \{-M', \dots, M' \}^d} \V \p{\phi_k(X_1)} \\
        &\stackrel{\Cref{lemmaProvicius}}{\leq} \frac{(2 M' +1)^d}{n} \;.
    \end{aligned}
\end{equation}
This last value may thus be used as $H^2$ in the application of \Cref{factTalagrandInequality}.

Furthermore, for any $g$ such that $\|g \| \leq 1$,
\begin{equation}
    \begin{aligned}
    \max\{ \| R(T_g^{(M')})\|_{\infty} , \| I(T_g^{(M')})\|_{\infty} \}
        &\leq\| T_g^{(M')}\|_{\infty} \\
        &= 
        \sup_t \left| \sum_{k \in \{-M', \dots, M' \}^d} \left\langle{\phi}_k , g \right\rangle \bar{\phi}_k(t) \right| \\
        &\stackrel{\text{Cauchy-Schwarz}}{\leq}
        \sup_t
        \sqrt{\sum_{k \in \{-M', \dots, M' \}^d} |\left\langle{\phi}_k , g \right\rangle|^2}
        \sqrt{\sum_{k \in \{-M', \dots, M' \}^d} |\bar{\phi}_k(t)|^2} \\
         &\stackrel{\|g\|\leq 1 \& |\bar{\phi}_k(\cdot)| \leq 1}{\leq}
         \sqrt{(2 M' + 1)^d} \\
         &\leq \sqrt{n}
    \end{aligned}
\end{equation}
where the last inequality comes from the fact that $(2 \max \mathcal{M} + 1)^d \leq n$.
This gives the value of $M_1$ for \Cref{factTalagrandInequality}.

Finally, for any $g$ such that $\|g \| \leq 1$,
\begin{equation}
    \begin{aligned}
    \max \{ \V(R(T_g^{(M')}(X_1))), \V(I(T_g^{(M')}(X_1)))\}
        &\leq \V(T_g^{(M')}(X_1)) \\
        &\leq \E\p{\left|T_g^{(M')}(X_1)\right|^2} \\
        &= \int \left|T_g^{(M')}(x)\right|^2 f(x) dx \\
        &\stackrel{|f(\cdot)| \| \leq f\|_{\infty} \text{ a.s.}}{\leq } \| f\|_{\infty} \int \left|T_g^{(M')}\right|^2  dx \\
        &\stackrel{\| T_g^{(K')}\| \leq 1}{\leq } \| f\|_{\infty}
    \end{aligned}
\end{equation}
which gives the value of $v$ for \Cref{factTalagrandInequality}.

So in the end, \Cref{factTalagrandInequality} tells us that there exists absolute constants $C_1, C_2, C_3 > 0$ such that, when tuned with $\Lambda^{(1)}_{\text{samp}}(K) \geq 96 \frac{(2K+1)^d}{n}$,
\begin{equation}
    \begin{aligned}
        \E \p{ \sum_{K' \in \mathcal{K}} \p{\|\tilde{f}_{M'} - {f}_{M'} \|^2 - \frac{\Lambda^{(1)}_{\text{samp}}(K')}{12}}_+}
        &\leq 
        \sum_{M' \in \mathcal{K}}
        \frac{C_1}{n}\p{\| f\|_{\infty} e^{-C_2 \frac{(2 M' +1)^d}{\| f\|_{\infty}}} + e^{- C_2 \sqrt{(2 M' +1)^d}}}
    \end{aligned}
\end{equation}

Hence, since the series $\sum_n e^{-n}$ and $\sum_n e^{-\sqrt{n}}$ converge, there exists a constant $C$ depending only on $\| f\|_{\infty}$ such that
\begin{equation}
    \begin{aligned}
        \E \p{ \sum_{M' \in \mathcal{M}} \p{\|\tilde{f}_{M'} - {f}_{M'} \|^2 - \frac{\Lambda^{(1)}_{\text{samp}}(M')}{12}}_+}
        &\leq 
        \frac{C}{n} \;.
    \end{aligned}
\end{equation}

\paragraph{Control of the extra term 3.}

For any $\mathcal{M} \in \mathcal{M}$,

\begin{equation}
\begin{aligned}
\E \p{ \max_{M' \in \mathcal{M}} \p{\|\hat{f}_{M'} - \tilde{f}_{M'} \|^2 - \frac{\Lambda^{(1)}_{\text{priv}}(M')}{12} }_+} 
&\leq 
\E \p{ \sum_{M' \in \mathcal{M}} \p{\| \hat{f}_{M'} - \tilde{f}_{M'} \|^2 - \frac{\Lambda^{(1)}_{\text{priv}}(M')}{12} }_+} \\
&=
\sum_{M' \in \mathcal{M}}  \E \p{ \p{\|  \hat{f}_{M'} - \tilde{f}_{M'} \|^2 - \frac{\Lambda^{(1)}_{\text{priv}}(M')}{12} }_+}
\end{aligned}
\end{equation}

For any $M'$, we may notice that $\| \hat{f}_{{M'}} - \tilde{f}_{{M'}}\|^2$ has a $\chi^2$ distribution scaled by $\sigma_{M'}$  and with $2(2M'+1)^d$ degrees of freedom. \Cref{factChiSquaredConcentration} using $\delta=1$ thus yields:
\begin{equation}
    \E\p{\p{\| \hat{f}_{{M'}} - \tilde{f}_{{M'}}\|^2 - (1+1) \sigma_{M'}^2 2(2M'+1)^d}_+ } \leq \frac{2 \sigma_{M'}^2}{1} e^{- \frac{2(2M'+1)^d }{4}} + 2 \sigma_{M'}^2 e^{- \frac{2(2M'+1)^d }{2}}
\end{equation}
Furthermore, since $\sigma_{M'} = \frac{2 \sqrt{(2M'+1)^d}}{n\sqrt{\rho'}}$, we have that  
\begin{equation}
    \E\p{\p{\| \hat{f}_{{M'}} - \tilde{f}_{{M'}}\|^2 - \frac{8 (2M'+1)^{2d}}{n^2 \rho'}}_+ } \leq C \frac{(2M'+1)^{d}}{n^2 \rho'} \p{e^{- \frac{(2M'+1)^d }{2}} + e^{- \frac{(2M'+1)^d }{1}}} \;,
\end{equation}
where $C$ is a non-negative absolute constant.

In the end, using that from our statement $\Lambda^{(1)}_{\text{priv}}(M') \geq  \frac{96 (2M'+1)^{2d}}{n^2 \rho'}$, we may write that
\begin{equation}
\label{likjhfkjdqsdfvfd}
    \begin{aligned}
        \E \p{ \max_{M' \in \mathcal{M}} \p{\|\hat{f}_{M'} - \tilde{f}_{M'} \|^2 - \frac{\Lambda^{(1)}_{\text{priv}}(M')}{12} }_+} 
        &\leq 
        \sum_{M' \in \mathcal{M}}  \E\p{\p{\| \hat{f}_{{M'}} - \tilde{f}_{{M'}}\|^2 - \frac{8 (2M'+1)^{2d}}{n^2 \rho'}}_+ }\\
        &\leq 
        \sum_{M' \in \mathcal{M}} C \frac{(2M'+1)^{d}}{n^2 \rho'} \p{e^{- \frac{(2M'+1)^d }{2}} + e^{- \frac{(2M'+1)^d }{1}}}  \\
        &\leq 
        \frac{C}{n^2 \rho'} \sum_{j \in \N} j \p{e^{- \frac{j }{2}} + e^{- \frac{j}{1}}} \\
        &\leq \frac{C'}{n^2 \rho'}
    \end{aligned}
\end{equation}
where $C'$ is a non-negative absolute constant since $\sum_{j \in \N} j \p{e^{- \frac{j }{2}} + e^{- \frac{j}{1}}}$ is finite.

\paragraph{Control of the extra term 1.}

For a fixed $\mathcal{M} \in \mathfrak{M}$,
\begin{equation}
    \begin{aligned}
        \E &\p{\p{\| \proj_{S_{M}}(\hat{f}_{\hat{M}})  - \proj_{S_{\hat{M}}}(\hat{f}_{{M}})\|^2 - (\Delta_{\text{priv}}(M) + \Delta_{\text{priv}}(\hat{M}))}_+} \\
        &\leq
        \E \p{\p{2\| \proj_{S_{{M}}}(\hat{f}_{\hat{M}}) - \tilde{f}_{\hat{M} \wedge M} \|^2 + 2\| \proj_{S_{\hat{M}}}(\hat{f}_{{M}}) - \tilde{f}_{\hat{M} \wedge M} \|^2 - (\Delta_{\text{priv}}(M) + \Delta_{\text{priv}}(\hat{M}))}_+} \\
        &\leq
        \E \p{\p{2\| \hat{f}_{\hat{M}} - \tilde{f}_{\hat{M}} \|^2 + 2\| \hat{f}_{{M}} - \tilde{f}_{{M}} \|^2 - (\Delta_{\text{priv}}(M) + \Delta_{\text{priv}}(\hat{M}))}_+} \\
        &\leq
        \E \p{\p{2\| \hat{f}_{\hat{M}} - \tilde{f}_{\hat{M}} \|^2 - \Delta_{\text{priv}}(\hat{M})}_+ +  2\| \hat{f}_{{M}} - \tilde{f}_{{M}} \|^2 } \\
        &\leq 
        \E \p{\sum_{K' \in \mathcal{M}}\p{2\| \hat{f}_{M'} - \tilde{f}_{M'} \|^2 - \Delta_{\text{priv}}(M')}_+ +  2\| \hat{f}_{{M}} - \tilde{f}_{{M}} \|^2 } \\
    \end{aligned}
\end{equation}

Furthermore, following a roadmap similar as the one used in the control of the extra term 3 (see \eqref{likjhfkjdqsdfvfd}), we observe that if $\Delta_{\text{priv}}(M') \geq  \frac{16 (2M'+1)^{2d}}{n^2 \rho'}$, there exists an absolute constant $C > 0$ such that 
\begin{equation}
    \E \p{\sum_{M' \in \mathcal{M}}\p{2\| \hat{f}_{M'} - \tilde{f}_{M'} \|^2 - \Delta_{\text{priv}}(M')}_+ +  2\| \hat{f}_{{M}} - \tilde{f}_{{M}} \|^2 } \leq C \p{\frac{1}{n^2 \rho'} + \| \hat{f}_{{M}} - \tilde{f}_{{M}} \|^2}
\end{equation}

\paragraph{Putting the pieces together.}

All in all, by taking the expectation, we have proved that for any $M \in \mathcal{M}$,
\begin{equation}
    \begin{aligned}
        \E \p{\|\hat{f}_{\hat{M}} - f \|^2} / C
        &\leq 
        \|f - {f}_{M}\|^2  + \E \p{\| \hat{f}_M - \tilde{f}_M\|^2} + \E \p{\sum_{M' \in \mathcal{M}} \p{\| \hat{f}_{M'} - \tilde{f}_{M'} \|^2 - \frac{\Lambda^{(1)}_{\text{priv}}(M')}{12} }_+}  \\ 
        &+ \E \p{\sum_{M' \in \mathcal{M}}\p{\| \hat{f}_{M'} - \tilde{f}_{M'} \|^2 - \frac{\Delta_{\text{priv}}(M')}{2}}_+ } +\E \p{ \sum_{M' \in \mathcal{M}} \p{\|\tilde{f}_{M'} - {f}_{M'} \|^2 - \frac{\Lambda^{(1)}_{\text{samp}}(M')}{12}}_+} \\
        &+ \Lambda^{(2)}(M)
    \end{aligned}
\end{equation}
where $C>0$ is an absolute constant

Furthermore, $\E \p{\| \hat{f}_M - \tilde{f}_M\|^2} = 2 (2M + 1)^d \sigma_M^2$, and with the values of $\Lambda^{(1)}_{\text{samp}}(\cdot)$, $\Lambda^{(1)}_{\text{priv}}(\cdot)$ and $\Delta_{\text{priv}}(\cdot)$ that were taken within the proof, the other expectations are controlled, yielding
\begin{equation}
    \begin{aligned}
        \E \p{\|\hat{f}_{\hat{M}} - f \|^2} / C'
        &\leq 
        \|f - {f}_{M}\|^2  +2 (2M + 1)^d \sigma_M^2 + \E \p{\sum_{M' \in \mathcal{M}} \p{\| \hat{f}_{M'} - \tilde{f}_{M'} \|^2 - \frac{\Lambda^{(1)}_{\text{priv}}(M')}{12} }_+}  \\ 
        &+ \E \p{\sum_{M' \in \mathcal{M}}\p{\| \hat{f}_{M'} - \tilde{f}_{M'} \|^2 - \frac{\Delta_{\text{priv}}(M')}{2}}_+ } +\E \p{ \sum_{M' \in \mathcal{M}} \p{\|\tilde{f}_{M'} - {f}_{M'} \|^2 - \frac{\Lambda^{(1)}_{\text{samp}}(M')}{12}}_+} \\
        &+ \Lambda^{(2)}(M)
    \end{aligned}
\end{equation}

\subsection{Proof of \Cref{theoremAdaptivePrivateUtility}}
\label{proofOfTheoremAdaptivePrivateUtility}

We recall that for any $M \in \mathcal{M}$, the bias-variance tradeoff $BV(M)$ in \Cref{theoremAdaptiveSelectionUtility} reads 
\begin{equation}
        \begin{aligned}
            BV(M) 
            \leq 
            &{\| f - f_M\|^2} + 
            {\frac{ (2M + 1)^d}{n}} + {2 (2M + 1)^d \sigma_M^2} \;,
        \end{aligned}
    \end{equation}
where $\sigma_M = \frac{2 \sqrt{(2M+1)^d}}{n\sqrt{\rho / |\mathcal{M}|}}$. 

As in the proof of \Cref{theoremUpperBoundSobolevClass}, the dichotomy of having the variance dominated by sampling or privacy leads to the the introduction of the optimal cut-off 
\begin{equation*}
    M^* + 1 \eqdef \min \left\{  (n/2^d) ^{\frac{1}{2 \beta + d}} , \p{n \sqrt{\rho / \mathcal{M}} / 2^d}^{\frac{1}{\beta + d}}  \right\} \;.
\end{equation*}
If one could guarantee that $M^* +1$ belongs to $\mathcal{M}$, then \Cref{theoremAdaptiveSelectionUtility} would guarantee the advertised result. However, this is not the case.

Even if one cannot guarantee that $M^* +1 \in \mathcal{M}$, with the construction rule for $\mathcal{M}$, we can always guarantee that for $n$ big enough (the "big enough" depends on $\beta$ and $d$), there will exist $M'+1 \in \mathcal{M}$ such that $(M^*+1)/2 \leq M' +1 \leq M^* +1$.

Since the variance terms are non-increasing with $M$, using $M'$ instead of $M^*$ only decreases the variance. 

The bias term on the other hand is non-decreasing with $M$. However, by looking at the expressions of the bias in \Cref{lemmaBiasUBSobolev} or \Cref{prop:Fourier_Sobolev_estimates} shows that in the worst case, being off by a factor at most $1/2$ degrades the estimation bias by a factor $2^{2 \beta}$.

Using that the $\min_{M \in \mathcal{M}} BV(M)$ in \Cref{theoremAdaptiveSelectionUtility} is upper-bounded by $BV(M')$ and that the residual terms are negligible yields the result.

\section{On non-integer multi-dimensional Sobolev spaces}  
\label{sec:realSobolevSpaces}

This section presents all the technical details on how to handle Sobolev spaces of non-integer smoothness.

\subsection{Definition}

Below, we shall discuss on the multi-dimensional Sobolev spaces with a non-integer parameter $\beta \ge 0$.

Our starting point is the space of Hölder functions with a (fractional) order $s \in (0,1)$ and radius $R$:
\begin{equation}
    \label{def:holder_space}
    \mathcal{H}_R(s)= \left\{f: \mathbb{R}^d \longrightarrow \mathbb{R} \, \vert \, \|f\|_{\mathcal{H}_s} := \sup_{(x,y) \in [0,1]^d \times [0,1]^d} \frac{|f(x)-f(y)|}{\|x-y\|^s} \leq R \right\}.
\end{equation}
Then, for any \textit{real} value $\beta$, we shall use the decomposition $\beta=\lfloor \beta \rfloor + \nu$ where $\nu = \beta-\lfloor \beta \rfloor \in [0,1)$. 
In this decomposition, $\lfloor \beta \rfloor$ is then the integer part of the order derivatives and $\nu$ the fractional one: $\lfloor \beta \rfloor$ encodes for a number of integer derivatives whereas $\nu$ refers to an Hölder smoothness of these derivatives. 

For a given $L>0$, we will say that $f \in S_L(\beta)$ if:
\begin{equation}
    \label{def:sobolev_fractional}
    S_L(\beta) := \left\{ f:\mathbb{R}^d \longrightarrow \mathbb{R} \, \vert \, \sum_{|\alpha|=\lfloor \beta \rfloor} \|\partial^{\alpha} f\|_2^2  + \mathbf{1}_{\nu >0} \sum_{|\alpha|=\lfloor \beta \rfloor} 
\|\partial^{\alpha} f\|_{\mathcal{H}_\nu}^2 \leq L^2 \right\}.
\end{equation}
We observe that when $\beta$ is an integer, $S_L(\beta)$ synchronises with the standard definition. 

\subsection{Control of the bias}

We establish below the important tail behaviour of the Fourier series in our generalized Hölder Sobolev spaces:
\begin{proposition}
    \label{prop:Fourier_Sobolev_estimates}
    Assume that $f \in  S_L^p(\beta)$, then an explicit constant $\square(\beta)$ independent from $d$ exists such that
     $$
     \sum_{k \notin \{-M,\ldots,M\}^d} |\theta_k(f)|^2 \leq \square(\beta) d (M+1)^{-2\beta} L^2 $$
\end{proposition}

We first state an important proposition on the relation between the fractional Holder exponent $s \in (0,1)$ of any function $f$ and the Fourier series associated to $f$.
\begin{proposition}
    \label{prop:Fourier_Holder}
    Assume that $f \in \mathcal{H}_R(s)$ for $s \in (0,1)$ and that $f$ satisfies the periodicity condition \eqref{equationPeriodicityCondition} for $\alpha = (0, \dots, 0)$, then the Fourier series associated to $(\theta_k(f))_{k \in \mathbb{Z}^d}$ of $f$ satisfies
    $$
    \sum_{k \notin \{-M,\ldots,M\}^d}  |\theta_k(f)|^2 \leq C(s) d R^2 (M+1)^{-2s},
    $$
    where 
    $
    C(s) = \frac{2^{2s} 3^{-s}}{1-2^{-2s}}.$
\end{proposition}

\begin{proof}
Below, $k$ refers to a $d$ dimensional vector of integers, and $\max(|k|)$ is the maximal value of the vector that contains the absolute values of the coordinates of $k$.

    We consider $f$ and a translation of $f$ denoted by $f_{\mathfrak{h}}$: $f_{\mathfrak{h}}(x) = f(x-\mathfrak{h})$ where $\mathfrak{h}$ is any vector of $[0,1]^d$. Using the periodicity of $f$, we have:
    $$
    \forall k \in \mathbb{Z}^d \qquad \theta_k(f) = \int_{[0,1]^d} f(x) e^{-i_{\C} 2 \pi \langle k, x \rangle} \text{d}x \qquad \text{and} \qquad \theta_k(f_{\mathfrak{h}}) = \int_{[0,1]^d} f_{\mathfrak{h}}(x) e^{-i_{\C} 2 \pi \langle k, x \rangle} \text{d}x
    = \theta_k(f) e^{-i_{\C} 2 \pi \langle k, \mathfrak{h} \rangle}, 
    $$
    which entails:
    $$
    \theta_k(f) \left(1-e^{-i_{\C} 2 \pi \langle k, \mathfrak{h} \rangle}\right) = \int_{[0,1]^d} (f(x)-f_{\mathfrak{h}}(x)) e^{-i_{\C} 2 \pi \langle k, x \rangle} \text{d}x 
    $$
We get from the Parseval equality and the fractional Holder hypothesis on $f$, for any collection of vectors $\mathfrak{h}^{(j)} \in [0,1]^d$:
\begin{equation}
\label{eq:parseval_collection}
\forall j \in \{1,\ldots,d\} \qquad 
    \sum_{k \in \mathbb{Z}^d} |\theta_k(f)|^2 \left| 1-e^{-i_{\C} 2 \pi \langle k, \mathfrak{h}^{(j)} \rangle}\right|^2 = \|f-f_{\mathfrak{h}^{(j)}}\|_2^2 \leq R^2 |\mathfrak{h}^{(j)}|^{2s}
\end{equation}
We now consider $k=(k_1,\ldots,k_d) \in \mathbb{Z}^d$ and assume that for $j \in \{1,\ldots,d\}: |k_j| = K \in [2^m,2^{m+1})$.
For this coordinate $j \in \{1,\ldots,d\}$, we consider the vector $\mathfrak{h}^{(j)} =  \frac{2^{-m}}{3} \delta_{j}$ and we verify that:
$$
|2 \pi \langle k, \mathfrak{h^{(j)}} \rangle| = \frac{2 \pi }{3} k_i 2^{-m} \in \left[\frac{2 \pi }{3},\frac{4 \pi }{3}\right).
$$
It implies that:
$$ \left| 1-e^{-i_{\C} 2 \pi \langle k, \mathfrak{h}^{(j)} \rangle}\right|^2 \ge 1,$$
which in turn leads to
\begin{align*}
\sum_{k \in \mathbb{Z}^d : |k_j| \in [2^m,2^{m+1})} |\theta_k(f)|^2
& \leq 
\sum_{k \in \mathbb{Z}^d : |k_j| \in [2^m,2^{m+1})}
 |\theta_k(f)|^2
\left| 1-e^{-i_{\C} 2 \pi \langle k, \mathfrak{h}^{(j)} \rangle}\right|^2
\\
& \leq R^2 |\mathfrak{h}^{(j)}|^{2s}
\end{align*}
where we applied Equation \eqref{eq:parseval_collection} in the last line. Using the value of $\mathfrak{h}^{(j)}$, we deduce that:
\begin{equation}
    \label{eq:borne_serie_j}
    \forall j \in \{1,\ldots,d\} \qquad 
    \sum_{k \in \mathbb{Z}^d : |k_j| \in [2^m,2^{m+1})} |\theta_k(f)|^2 \leq R^2 3^{-2s} 2^{-2ms}.
\end{equation}

We are now able to conclude the proof: we consider any integer $M \ge 1$ and the dyadic scale, which associates $m_0 \ge 0$ such that $2^{m_0} \leq M < 2^{m_0+1}$, we observe that
\begin{align*}
    \sum_{k \in \mathbb{Z}^d \,:\, k \notin \{-M,\ldots,M\}^d}
    |\theta_k(f)|^2     & \leq \sum_{k \in \mathbb{Z}^d \,:\, k \notin \{-2^{m_0},\ldots,2^{m_0}\}^d}
    |\theta_k(f)|^2 \\
    & \leq \sum_{k \in \mathbb{Z}^d \, \exists j \, : |k_j|\ge 2^{m_0}}
    |\theta_k(f)|^2 \\
    & \leq \sum_{j=1}^d \sum_{k \in \mathbb{Z}^d \, :   |k_j|\ge 2^{m_0}}
    |\theta_k(f)|^2 \\
    & \leq \sum_{j=1}^d \sum_{m \ge m_0} \sum_{k \in \mathbb{Z}^d \, :   2^{m} \leq |k_j| < 2^{m+1}}
    |\theta_k(f)|^2 \\
    & \leq R^2 3^{-s}  \sum_{j=1}^d \sum_{m \ge m_0}  2^{-2ms} \\
    & \leq \frac{R^2 3^{-s}}{1-2^{-2s}} d 2^{-2 m_0 s} \\
    & \leq \frac{R^2 2^{2s} 3^{-s}}{1-2^{-2s}} d (M+1)^{-2s}.
    \end{align*}
We obtain the conclusion of the proof with $C(s)=\frac{ 2^{2s} 3^{-s}}{1-2^{-2s}}.$
\end{proof}

\begin{proof}[Proof of Proposition \ref{prop:Fourier_Sobolev_estimates}]
    
We are now ready to extend our estimate stated in Lemma \ref{lemmaBiasUBSobolev} from integer Sobolev spaces to fractional ones.
Assume that $\beta>0$: we observe that
\begin{itemize}
    \item If $\beta \in \mathbb{N}$, then Lemma \ref{lemmaBiasUBSobolev} yields
    $$
    \|f-f_M\| \leq \frac{L^2}{(2\pi)^{2 \beta}} (M+1)^{-2 \beta}.
    $$
    \item Oppositely, if $\beta=\lfloor \beta \rfloor + s$ with $s \in (0,1)$ and assume that $f \in S_L^p(\beta)$, we know from Proposition \ref{prop:Fourier_Holder} that:
    $$
    \sum_{|\alpha| = \lfloor \beta \rfloor} \sum_{k \notin \{-M,\ldots,M\}^d} |\theta_k(\partial^{\alpha}(f))|^2 \leq 
    \sum_{|\alpha| = \lfloor \beta \rfloor} C(s) d  M^{-2s} \|\partial^{\alpha} f\|_{\mathcal{H}_s}^2 \leq 
     C(s) d (M+1)^{-2s} L^2.
    $$
    We then conclude following the same guidelines as the ones of 
    Lemma \ref{lemmaBiasUBSobolev}:
    \begin{align*}
((2\pi)(M+1))^{2 \lfloor \beta \rfloor }  \sum_{k \notin \{-M,\ldots,M\}^d} |\theta_k(f)|^2& \leq 
\sum_{|\alpha| = \lfloor \beta \rfloor} \sum_{k \notin \{-M,\ldots,M\}^d} (2 \pi k)^{2 \alpha} |\theta_k(f)|^2 \\
& \leq C(s) d (M+1)^{-2s} L^2,
    \end{align*}
    which implies with $\beta=\lfloor \beta \rfloor + s$ the final bound:
    $$
     \sum_{k \notin \{-M,\ldots,M\}^d} |\theta_k(f)|^2 \leq \frac{C(s)L^2 }{(2\pi)^{2 \lfloor \beta \rfloor}} d (M+1)^{-2\beta}.
    $$
\end{itemize}
\end{proof}

\subsection{Lower-bounds : Adaptation of the proof of \Cref{theoremLowerBoundSobolev} in the case of real-valued $\beta$'s}
\label{sec:adaptation_proof_lower_bound}

The only adaptation needed to the proof is to handle the new Hölder part in the definition. In fact, the only adaptation needed is to slightly modify the function $\psi (\cdot)$ in \Cref{proofOfTheoremLowerBoundSobolev}, and to verify that the subsequent family of functions defined from it is a family of densities of probability in $\mathcal{S}_L^p(\beta)$. We use the decomposition $\beta=\lfloor \beta \rfloor + \nu$ where $\nu = \beta-\lfloor \beta \rfloor \in [0,1)$.

Let $\epsilon >0$ that will be fixed later. The old $\psi$ of \Cref{proofOfTheoremLowerBoundSobolev} is replaced by a new $\psi(\cdot) = a \Psi\p{\frac{\cdot}{2}}$ where $a>0$ is fixed to a small enough value such that 
\begin{equation}
    \sum_{|\alpha|=\lfloor \beta \rfloor} \|\partial^{\alpha} \psi\|_2^2  + \mathbf{1}_{\nu >0} \sum_{|\alpha|=\lfloor \beta \rfloor} 
\|\partial^{\alpha} \psi\|_{\mathcal{H}_\nu}^2 \leq \epsilon\;.
\end{equation}
All the other quantities are defined from this new $\psi$ as in \Cref{proofOfTheoremLowerBoundSobolev}. 

The entire proof of \Cref{theoremLowerBoundSobolev} remains unchanged except for one detail : we first to check that the new family of densities $(f_{\theta})$ is in $\mathcal{S}_L^p(\beta)$ for non-integer $\beta$'s. We separate two cases :
\begin{itemize}
    \item \emph{$\beta \geq 1$} : Let $\theta \in \{ 1, \dots, m^d\}$, and let $|\alpha| = \lfloor \beta \rfloor$. With the same reasoning steps as in \eqref{kjhbdskqjfhsbkjhcvbkqsdf}, we obtain that 
    \begin{equation}
\begin{aligned}
    \|\partial^{\alpha} f_{\theta}\|_2^2
    &= \int_{[0, 1]^d} \p{h^{\beta} \sum_{i=1}^{m^d} \theta_i \p{x \mapsto \psi\p{\frac{x - p_i}{h}}}^{(\alpha)}}^2 \\
    &= \int_{[0, 1]^d} h^{2 \nu} \p{\sum_{i=1}^{m^d} \theta_i \psi^{(\alpha)}\p{\frac{\cdot - p_i}{h}}}^2 \\
    &\stackrel{\text{disjoint supports}}{=}  h^{2 \nu}\sum_{i=1}^{m^d} \theta_i \int_{[0, 1]^d} \p{ \psi^{(\alpha)}\p{\frac{\cdot - p_i}{h}}}^2 \\
    &\stackrel{\|\theta\|_1 \leq m^d \& \text{ variable swap}}{\leq} h^{2 \nu} m^d h^d \int_{[0, 1]^d} \p{\psi^{(\alpha)}}^2  \\
    &\stackrel{m^d h^d \leq 1}{\leq} h^{2 \nu}\|\partial^{\alpha} \psi\|_2^2
    \stackrel{h \leq 1}{\leq} \|\partial^{\alpha} \psi\|_2^2\;.
\end{aligned}
\end{equation}

In order to control $\|\partial^{\alpha} f_{\theta}\|_{\mathcal{H}_\nu}^2$, we will need the following lemma :
\begin{lemma}
\label{lemma:Hölder_Disjoint_Support}
    If $g_1$ and $g_2$ are continuous with compact supports and if their supports are disjoint, then
    \begin{equation*}
        \|g_1 + g_2\|_{\mathcal{H}_\nu} \leq \max \left\{\|g_1 \|_{\mathcal{H}_\nu}, \|g_2\|_{\mathcal{H}_\nu} \right\} \;.
    \end{equation*}
\end{lemma}
\begin{proof}
    Let $x \neq y \in [0, 1]^d$. We will upper-bound the Hölder ratio $\frac{|(g_1 + g_2)(x)-(g_1 + g_2)(y)|}{|x-y|^{\nu}}$ by a Hölder ratio depending only on $g_1$ or $g_2$. If $x$ and $y$ both live in the support of either $g_1$ or $g_2$, then we may rewrite, in the case where it is in the support of $g_1$,
    \begin{equation}
        \frac{|(g_1 + g_2)(x)-(g_1 + g_2)(y)|}{|x-y|^{\nu}} \leq \frac{|g_1(x)-g_1(y)|}{|x-y|^{\nu}} \leq \|g_1 \|_{\mathcal{H}_\nu} \;.
    \end{equation}
    Alternatively, the case when it is in the support of $g_2$ gives the majoration by $\|g_2 \|_{\mathcal{H}_\nu}$.

    Now let us look at the case where $x$ and $y$ do not both live in the support of either $g_1$ or $g_2$. Let us suppose that $g_1(x) \geq g_2(y)$, the other case being treated in the same fashion. Since $g_1$ and $g_2$ have disjoint supports, there exists $t \in (0, 1)$ such that $g_1(t x + (1 - t) y ) = g_2(t x + (1 - t) y ) = 0$ (connexity argument). Now, by the intermediate values theorem ($g_1$ is continuous), there exists $t' \in [0, t]$ such that $g_1(t' x + (1 - t') y ) = g_2(y)$. We thus obtain that  
    \begin{equation}
    \begin{aligned}
        \frac{|(g_1 + g_2)(x)-(g_1 + g_2)(y)|}{|x-y|^{\nu}} 
        &=
        \frac{|g_1(x)-g_2(y)|}{|x-y|^{\nu}} \\
        &=
        \frac{|g_1(x)-g_1(t' x + (1 - t') y )|}{|x-y|^{\nu}} \\
        &\leq 
        \frac{|g_1(x)-g_1(t' x + (1 - t') y )|}{|x-(t' x + (1 - t') y )|^{\nu}} \\
        &\leq 
        \|g_1 \|_{\mathcal{H}_\nu}\;.
    \end{aligned}
    \end{equation}
    The other case leads to a majoration by $\|g_2 \|_{\mathcal{H}_\nu}$.
    All in all, this proves that for any $x \neq y$,
    \begin{equation*}
        \frac{|(g_1 + g_2)(x)-(g_1 + g_2)(y)|}{|x-y|^{\nu}}  \leq \max \left\{\|g_1 \|_{\mathcal{H}_\nu}, \|g_2\|_{\mathcal{H}_\nu} \right\} \;,
    \end{equation*}
    and taking the supremum on the left-hand side yields the desired result.
\end{proof}

Back to our problem, we may write that
\begin{equation}
\begin{aligned}
    \|\partial^{\alpha} f_{\theta}\|_{\mathcal{H}_\nu}
    &= \left\|h^{\beta} \sum_{i=1}^{m^d} \theta_i \p{x \mapsto \psi\p{\frac{x - p_i}{h}}}^{(\alpha)}\right\|_{\mathcal{H}_\nu} \\
    &= h^{\nu} \left\|  \sum_{i=1}^{m^d} \theta_i \psi^{(\alpha)}\p{\frac{\cdot - p_i}{h}}\right\|_{\mathcal{H}_\nu} \\
    &\stackrel{\Cref{lemma:Hölder_Disjoint_Support}}{\leq}  h^{\nu} \max_{i} \left\{ \left\|   \psi^{(\alpha)}\p{\frac{\cdot - p_i}{h}}\right\|_{\mathcal{H}_\nu} \right\} \\
    &= h^{\nu} \max_{i} \left\{ \sup_{x \neq y} \frac{\psi^{(\alpha)}\p{\frac{x - p_i}{h}} - \psi^{(\alpha)}\p{\frac{y - p_i}{h}}}{\left| x - y \right|^{\nu}}\right\} \\
     &= h^{\nu} \max_{i} \left\{ h^{- \nu} \sup_{x \neq y}  \frac{\psi^{(\alpha)}\p{x} - \psi^{(\alpha)}\p{y}}{\left| x - y \right|^{\nu}}\right\} \\
     &= \|\partial^{\alpha} \psi\|_{\mathcal{H}_\nu}\;.
\end{aligned}
\end{equation}

So all in all, fixing $\epsilon = L^2$ ensures that for any $\theta$,
\begin{equation}
    \sum_{|\alpha|=\lfloor \beta \rfloor} \|\partial^{\alpha} f_{\theta}\|_2^2  + \mathbf{1}_{\nu >0} \sum_{|\alpha|=\lfloor \beta \rfloor} 
\|\partial^{\alpha} f_{\theta}\|_{\mathcal{H}_\nu}^2 \leq L^2\;.
\end{equation}

\item $\beta \in (0, 1)$ : When $\beta < 1$, there is one extra technical detail to consider : Since no integer derivative is performed, the constant parts in the densities $(f_{\theta})$ do not vanish. This is not a problem for the Hölder part since the seminorm $\|\cdot\|_{\mathcal{H}_\nu}$ is unchanged up to the addition or removal of a constant function. For the sobolev par on the other hand, we may use that $\| g_1 + g_2\|^2 \leq \p{1 + \eta} \| g_1 \|^2 + \p{1 + 1 / \eta} \|g_2\|^2$ for any $g_1, g_2 \in L^2$ and any $\eta>0$, which gives that 
\begin{equation}
    \|f_{\theta}\|_2^2  + 
\|f_{\theta}\|_{\mathcal{H}_\nu}^2 \leq L^2\;.
\end{equation}
when applied with $g_1$ the constant part of $f_{\theta}$, $g_2$ the part with the kernels, $\eta = \frac{L^2-1}{2}$, and  $\epsilon$ that satisfies $\p{1 + 2 / \eta} \epsilon \leq \frac{L^2 - 1}{2}$. Obviously, this only holds if $L>1$. However, since $\beta \in (0, 1)$, Jensen's inequality already implies that any density of probability $g$ satisfies $\|g\|_2^2 + \|g\|_{\mathcal{H}_\nu}^2 \geq \|g\|_2^2 \geq 1$, with equality if and only if $g$ is the density of the uniform distribution. So $L>1$ is not restrictive on non-trivial classes of distributions $\mathcal{S}_L^p(\beta)$.

Now that we have verified that the family of densities $(f_{\theta})$ is a subset of $\mathcal{S}_L^p(\beta)$, the rest of the proof follows line by line the one of \Cref{theoremLowerBoundSobolev} in the case of integer-valued $\beta$.

\end{itemize}

\section{Technical results}



\begin{lemma}[Popoviciu's inequality for multivariate random variables]
\label{lemmaProvicius}
    Let $X$ be a random variable in $\R^{d'}$. If there exist $\mu$ and $\sigma$ such that $\|X - \mu \| \leq \sigma$ almost-surely, then one has 
    \begin{equation}
        \V(X) \eqdef \E \p{ \|X - \E (X) \|^2} \leq \sigma^2 \;,
    \end{equation}
    thus allowing to gain a factor $4$ compared to the natural majoration $\V(X) \leq 4 \sigma^2$.
    In particular, with the isometric identification $(\C, | \cdot |) \cong (\R^2, \| \cdot \|)$, this allows bounding the variance of a complex random variable.
\end{lemma}
\begin{proof}
    $\V(X)$ minimizes the function $t \mapsto \E \p{\|X - t \|^2}$. Thus, $\V(X)$ is upper-bounded by the value of the same function in $\mu$, yielding the result. 
\end{proof}

\begin{lemma}[Existence of $\mathcal{C}^{\infty}$ function with support in unit ball of $\R^d$]
\label{existencecompactsupport}
    The function $\Psi$ from $R^d$ to $[0, + \infty)$ which is defined by
    \begin{equation}
        \Psi (x) \eqdef \begin{cases}
e^{- \frac{1}{1 - \|x\|^2}} \quad \text{if } \|x \| < 1\\
0 \quad \text{otherwise }
\end{cases}
    \end{equation}
is in $\mathcal{C}^{\infty}(\R^d)$ and takes non-negative values. 
\end{lemma}
\begin{proof}
    By induction, we get that for any $\alpha \in \N^d$, $\partial^{\alpha} \phi (x) = \frac{P_{\alpha}(x)}{Q_{\alpha}(x)} e^{- \frac{1}{1 - \|x\|^2}}$ when $\|x\| < 1$ where $P_{\alpha}$ and $Q_{\alpha}$ are polynomial expressions (in the coefficients of their input vector) with $Q_{\alpha}(x) \neq 0$, and immediately $\partial^{\alpha} \Psi(x) = 0$ when $\|x\| > 1$. This proves that $\partial^{\alpha} \Psi$ is continuous on $\R^d$ with $\partial^{\alpha} \Psi(x) = 0$ when $\|x\| \geq 1$ because the exponential term is dominant near the unit circle. Since this holds for any $\alpha \in \N^d$, the result follows.
\end{proof}

\begin{lemma}[Hoeffding's inequalities]
    If $X_1, \dots, X_n$ are independent real-valued random variables such that for any $i$, $a_i \leq X_i \leq b_i$, then for any $t > 0$,
    \begin{equation*}
        \Prob \p{ \left|\sum_i (X_i - \E(X_i)) \right| > t}  \leq 
        2 \exp \p{- \frac{2 t^2}{\sum_i (b_i - a_i)^2}} \;.
    \end{equation*}
    As a consequence, if $X_1, \dots, X_n$ are independent complex-valued random variables such that for any $i$, $X_i \in B(c_i, r_i)$,
    \begin{equation*}
        \Prob \p{\left|\sum_i (X_i - \E(X_i))\right| > t} \leq 4 \exp \p{- \frac{ t^2}{4 \sum_i r_i^2}} \;.
    \end{equation*}
\end{lemma}
\begin{proof}
    The first inequality for real-valued random variables is folklore, and its proof may for instance be found in \cite{tsybakov2003introduction}. For the claim about complex random variables, we have 
    \begin{equation*}
        \begin{aligned}
             \Prob &\p{\left|\sum_i (X_i - \E(X_i))\right| > t} \\
             &= \Prob \p{\left|\sum_i (X_i - \E(X_i))\right|^2 > t^2} \\
             &= \Prob \p{R\p{\sum_i (X_i - \E(X_i))}^2 + I\p{\sum_i (X_i - \E(X_i))}^2 > t^2} \\
             &\leq
             \Prob \p{R\p{\sum_i (X_i - \E(X_i))}^2 > t^2/2}
             +
             \Prob \p{I\p{\sum_i (X_i - \E(X_i))}^2 > t^2/2} \\
             &=
             \Prob \p{\p{\sum_i (R(X_i) - \E(R(X_i)))}^2 > t^2/2}
             +
             \Prob \p{\p{\sum_i (I(X_i) - \E(I(X_i)))}^2 > t^2/2} \\
             &\leq 
             2 \exp \p{- \frac{2 (t/\sqrt{2})^2}{\sum_i (2 r_i)^2}} + 2 \exp \p{- \frac{2 (t/\sqrt{2})^2}{\sum_i (2 r_i)^2}} \;,
        \end{aligned}
    \end{equation*}
    where the last inequality comes from Hoeffding's inequality for real-valued random variables.
\end{proof}

\begin{lemma}[$\chi^2$ concentration]
\label{factChiSquaredConcentration}
    Let $X_1, \dots, X_d$ be i.i.d. random variables with distribution $\mathcal{N}(0, \sigma^2)$. Let us define $Z = X_1^2 + \dots + X_d^2$. Then, for any $\delta > 0$,
    \begin{equation}
        \Prob \p{Z \geq (1+\delta) d \sigma^2} \leq \max \left\{ e^{- \frac{d \delta^2}{4}} , e^{- \frac{d \delta}{2}}\right\}  \;.
    \end{equation}
    Furthermore, the integrated version gives, for any $\delta > 0$,
    \begin{equation}
        \E \p{\p{Z - (1+\delta) d \sigma^2}_+} \leq  \frac{2 \sigma^2}{\delta} e^{- \frac{d \delta^2}{4}} + 2 \sigma^2 e^{- \frac{d \delta}{2}}\;.
    \end{equation}
\end{lemma}
\begin{proof}
    According to Lemma 1 in \cite{laurent2000adaptive}, for any $x > 0$, 
    \begin{equation}
        \Prob \p{Z \geq d \sigma^2  + 2 \sigma^2 \sqrt{d x} + 2 \sigma^2 x} \leq e^{-x} \;.
    \end{equation}
    Furthermore, we have $\delta \sigma^2 d = 2 \sigma^2 \sqrt{d x_1}$ iff $x_1 = \frac{d \delta^2}{4}$ and $\delta \sigma^2 d = 2 \sigma^2 x_2$ iff $x_2 = \frac{d \delta}{2}$. By noting $f(x) = 2 \sigma^2 \sqrt{d x} + 2 \sigma^2 x$, we have 
    \begin{equation}
    \begin{aligned}
        \Prob \p{Z \geq (1+\delta) d \sigma^2} &\leq \Prob \p{Z \geq d \sigma^2  +  f(\min \{ x_1, x_2\})}  \\
        &\leq e^{- \min \{ x_1, x_2\}} \\
        &= \max \left\{ e^{- \frac{d \delta^2}{4}} , e^{- \frac{d \delta}{2}}\right\} \;.
    \end{aligned}
    \end{equation}

    Furthermore,
     \begin{equation}
     \begin{aligned}
        \E \p{\p{Z - (1+\delta) d \sigma^2}_+} &\leq  
        \int_{(1+\delta) d \sigma^2}^{+ \infty} \Prob (Z \geq t) dt \\ 
        &= \int_{\delta}^{+ \infty} d \sigma^2\Prob (Z \geq (1+u) d \sigma^2) du \\
        &\leq \int_{\delta}^{+ \infty} d \sigma^2  \p{e^{- \frac{d u^2}{4}} + e^{- \frac{d u}{2}}} du \\
        &\leq \int_{\delta}^{+ \infty} d \sigma^2  \p{ \frac{u}{\delta} e^{- \frac{d u^2}{4}} + e^{- \frac{d u}{2}}} du \\
        &\stackrel{}{\leq} \frac{2 \sigma^2}{\delta} e^{- \frac{d \delta^2}{4}} + 2 \sigma^2 e^{- \frac{d \delta}{2}}\;.
    \end{aligned}
    \end{equation}
\end{proof}

\begin{lemma}[Talagrand's inequality (one of many) (From Appendix A in \cite{Non_Paramatric_Comte_2017})]
\label{factTalagrandInequality}
Let $n \in \N \setminus \{0\}$, $\mathcal{F}$ be a countable family of real-valued measurable functions and $(X_i)_{i = 1 \dots, n}$ be $n$ \emph{independent} random variables taking values in a common Polish space. By noting, for any $f \in \mathcal{F}$,
\begin{equation}
    \nu_n(f)  \eqdef \frac{1}{n} \sum_{i=1}^n (f(X_i) - \E(f(X_i))) \;,
\end{equation}
if there exist three positive constants $M_1$, $H$ and $v$ such that
\begin{equation}
    \sup_{f \in \mathcal{F}} \| f \|_{\infty} \leq M_1 \;,
\end{equation}
\begin{equation}
    \E \p{\sup_{f \in \mathcal{F}} |\nu_n(f)|} \leq H \;,
\end{equation}
\begin{equation}
    \sup_{f \in \mathcal{F}} \frac{1}{n} \sum_{i=1}^n \V(f(X_i)) \leq v \;,
\end{equation}
then for any $\delta > 0$,
\begin{equation}
    \begin{aligned}
        \E \p{\p{\sup_{f \in \mathcal{F}} |\nu_n(f)|^2 - 2 (1+2 \delta ) H^2}_+}
        &\leq
        \frac{4}{K_1} \p{\frac{v}{n} e^{- K_1 \delta \frac{n H^2}{v}} + \frac{49 M_1^2}{K_1 K(\delta)^2 n^2}  e^{- \frac{\sqrt{2} K_1 K(\delta) \sqrt{\delta}}{7} \frac{n H}{M_1}}} \;,
    \end{aligned}
\end{equation}
where $(y)_+ \eqdef \max\{y, 0 \}$, $K_1 \eqdef \frac{1}{6}$ and $K(\delta) \eqdef \min \{ \sqrt{1+ \delta} - 1, 1\} $.
\end{lemma}

\end{document}